\documentclass{amsart}
\usepackage[utf8]{inputenc}
\usepackage{amsmath, amssymb, amsthm, bm, bbm}
\usepackage{tikz-cd}
\usepackage{tikz}
\usepackage{minibox}
\usepackage[margin = 2cm]{geometry}
\usepackage[pdftex, colorlinks, linkcolor = blue, urlcolor = blue, citecolor = blue]{hyperref}

\usepackage{mathtools}
\usepackage{mhchem}

\newcommand{\myrightleftarrows}[1]{\mathrel{\substack{\xrightarrow{#1} \\[-.9ex] \xleftarrow{#1}}}}

\newtheorem{theorem}{Theorem}[section]
\newtheorem{proposition}[theorem]{Proposition}
\newtheorem{lemma}[theorem]{Lemma}

\newtheorem{corollary}[theorem]{Corollary}
\theoremstyle{definition}

\newtheorem{remark}[theorem]{Remark}
\theoremstyle{remark}
\newtheorem{definition}[theorem]{Definition}
\newtheorem*{claim}{Claim $(\ast)$}
\newtheorem*{proofofextensiontheorem}{Proof of Theorem~\ref{vermaexttheorem}}
\numberwithin{equation}{section}

\newcommand{\C}{\mathbb{C}}
\newcommand{\Z}{\mathbb{Z}}
\newcommand{\N}{\mathbb{N}}
\newcommand{\g}{\mathfrak{g}}
\newcommand{\h}{\mathfrak{h}}
\newcommand{\n}{\mathfrak{n}}
\newcommand{\p}{\mathfrak{p}}
\renewcommand{\sl}{\mathfrak{sl}}
\renewcommand{\r}{\mathfrak{r}}
\renewcommand{\l}{\mathfrak{l}}
\newcommand{\z}{\mathfrak{z}}
\newcommand{\m}{\mathfrak{m}}
\newcommand{\Sf}{\mathfrak{S}}
\renewcommand{\b}{\mathfrak{b}}

\newcommand{\calO}{\mathcal{O}}
\renewcommand{\O}{\mathcal{O}}
\newcommand{\X}{\mathcal{X}}
\newcommand{\Y}{\mathcal{Y}}

\newcommand{\lmod}{\operatorname{-mod}}

\newcommand{\Hom}{\operatorname{Hom}}
\newcommand{\rank}{\operatorname{rank}}
\newcommand{\dimension}{\operatorname{dim}}
\newcommand{\vspan}{\operatorname{span}}
\newcommand{\im}{\operatorname{im}}
\newcommand{\id}{\operatorname{id}}
\newcommand{\gr}{\operatorname{gr}}

\newcommand{\supp}{\operatorname{supp}}
\newcommand{\wt}{\operatorname{wt}}
\newcommand{\ch}{\operatorname{ch}}
\newcommand{\ad}{\operatorname{ad}}

\newcommand{\Ind}{\operatorname{Ind}}
\newcommand{\Ann}{\operatorname{Ann}}
\newcommand{\Id}{\operatorname{\mathbbm{1}}}
\newcommand{\ev}{\operatorname{ev}}
\newcommand{\Lie}{\operatorname{Lie}}

\def\onto{\twoheadrightarrow}
\def\isoto{\overset{\sim}{\longrightarrow}}

\title{Category $\calO$ for Truncated Current Lie Algebras}
\author{Matthew Chaffe and Lewis Topley}
\email{mxc167@student.bham.ac.uk} 

\email{lt803@bath.ac.uk} 

\address{School of Mathematics, University of Birmingham, Edgbaston, Birmingham, B15 2TT, UK.}
\address{Department of Mathematical Sciences, University of Bath, North Road, Claverton Down, Bath BA2 7AY, UK.}

\address{}

\begin{document}
\maketitle

\begin{abstract}
In this paper we study an analogue of the Bernstein--Gelfand--Gelfand category $\O$ for truncated current Lie algebras $\g_n$ attached to a complex semisimple Lie algebra. This category admits Verma modules and simple modules, each parametrised by the dual space of the truncated currents on a choice of Cartan subalgebra in $\g$. Our main result describes an inductive procedure for computing composition multiplicities of simples inside Vermas for $\g_n$, in terms of similar composition multiplicities for $\l_{n-1}$ where $\l$ is a Levi subalgebra. As a consequence, these numbers are expressed as integral linear combinations of Kazhdan--Lusztig polynomials evaluated at 1. This generalises recent work of the first author, where the case $n=1$ was treated.
\end{abstract}

\section{Introduction}

Truncated current Lie algebras have appeared in numerous parts of the literature in recent years, and a large part of their interest stems from the fact that they interpolate between the finite dimensional simple Lie algebras and the vacuum parabolic of the corresponding untwisted affine Lie algebra. If $\g = \Lie(G)$ is the Lie algebra of a complex reductive algebraic group then the group $G[t]$ of polynomial currents in $G$ is the (infinite dimensional) algebraic group of regular maps $\mathbb{A}^1_\C \to G$. The {\it current Lie algebra} $\g[t] := \Lie G[t] = \g \otimes \C[t]$ is isomorphic to the derived subalgebra of a maximal parabolic of the Kac--Moody affinisation of $\g$, and the {\it $n$th  truncated current Lie algebra} is the quotient $\g_n := \g \otimes \C[t]/(t^{n+1})$. Equivalently, $\g_n$ can be described as the Lie algebra of the $n$th jet scheme $J_n G$ of $G$. \vspace{2pt}

The first truncated currents Lie algebra $\g_1$ appeared in the work of Takiff, and so they are often referred to as {\it Takiff Lie algebras}. He showed that the symmetric invariant algebra $S(\g_1)^{\g_1}$ is a polynomial algebra on $2\rank(\g)$ variables \cite{T}, generalising the classical theorem of Chevalley which describes $S(\g_0)^{\g_0}$. Later Ra{\"i}s--Tauvel extended Takiff's theorem for arbitrary $n$ \cite{RT}. Since the latter work is a crucial ingredient in our paper, we will briefly describe their main results in Section~\ref{ss:centreofenvelopingalg}. More recently, Macedo--Savage extended their work further to the case of truncated multicurrents \cite{MSa}, and  Panyushev--Yakimova showed that for complex Lie algebras the operation $\g \mapsto \g_1$ preserves the property of having polynomial symmetric invariants, under mild assumptions \cite{PY}. \vspace{2pt}

These works on invariant theory have important applications in the theory of vertex algebras. Notably, the first approximation to describing the centre of the critical level universal affine vertex algebra associated to $\g$ (a.k.a. the Feigin--Frenkel centre) is the description of the semi-classical limit, which is equal to $S(\g[t^{-1}]t^{-1})^{\g[t]}$. In turn can be described as the direct limit of the algebras $S(\g_n)^{\g_n}$ discussed in the previous paragraph (see \cite[\textsection 3.4]{Fr} or \cite[\textsection 6]{Mo} for example). This connection with the Feigin--Frenkel centre was recently used by Arakawa--Premet to provide a positive solution to Vinberg's problem for centralisers, using affine $W$-algebras \cite{AP}. We also mention the work of Kamgarpour \cite{Ka} for a discussion of the relationship with the geometric Langlands program. Another connection with the theory of $W$-algebras is given by the work of Brundan--Kleshchev \cite[\textsection 12]{BK}, which states that the finite $W$-algebra for $\mathfrak{gl}_N$ associated to a nilpotent element with all Jordan blocks of size $n$ is isomorphic to a truncated Yangian, which admits $U(\g_n)$ as a filtered degeneration. \vspace{2pt}

This article focuses on the representation theory of truncated current algebras. The most famous category of modules for a complex reductive Lie algebra $\g$ is the Bernstein--Gelfand--Gelfand (BGG) category $\O$, which is an abelian category containing all highest weight modules (see \cite{H} for a survey). Wilson extended the notion of highest weight modules to all truncated currents \cite{W}, and subsequently Mazorchuck--Sörderberg introduced a version of category $\O$ for Takiff $\sl_2$ \cite{MS}. The most recent development in this field is the work of the first author \cite{Ch} which made a thorough study of category $\O$ for all Takiff Lie algebras, eventually showing that the composition multiplicities of simple modules inside Verma modules can be determined by certain formulas involving Kazhdan--Lusztig polynomials. The results of the present paper generalise all of the main results of {\it loc. cit.} to the case of truncated current Lie algebras $\g_n$.\vspace{2pt}

For the rest of the introduction we fix $\g = \Lie(G)$ where $G$ is a complex reductive algebraic group and fix $n > 0$. We also fix a triangular decomposition $\g = \n^- \oplus \h \oplus \n^+$ and write $\b = \h \oplus \n^+$. This gives rise to a triangular decomposition $\g_n = \n^-_n \oplus \h_n \oplus \n^+_n$, and we say that a module is {\it highest weight} if it is generated (as a $\g_n$-module) by a one dimensional $\b_n$-module. These one dimensional modules are parametrised by $\h_n^*$, and the one dimensional module afforded by $\lambda \in \h_n^*$ is denoted $\C_\lambda$. We define the universal highest weight module or {\it Verma module} of weight $\lambda$ by
\begin{eqnarray}
M_\lambda := U(\g_n) \otimes_{U(\b_n)} \C_\lambda.
\end{eqnarray}
Since this is semisimple over $\h = \h \otimes 1 \subseteq \h_n$, with one dimensional $\lambda|_\h$-weight space, it follows that it has a unique maximal submodule and a unique simple quotient, which we denote $L_\lambda$.\vspace{2pt}

We study the category $\O(\g_n)$ of finitely generated $\g_n$-modules on which $\n_n^+$ acts locally nilpotently, $\h$ acts semisimply and $\h_n^{\ge 1} :=  \h \otimes t\C[t] / (t^{n+1}) \subseteq \h_n$ acts locally finitely (see Definition~\ref{catOdef}). This category is closed under quotients and submodules, and the simple modules are precisely $\{L_\lambda \mid \lambda \in \h_n^*\}$.
\vspace{2pt} 

One especially nice feature of the BGG category $\O$ for $\g$ it that it is artinian, however this fails for $\O(\g_n)$. Therefore we define composition multiplicities $[M : L_\mu]$ using formal characters: every module $M \in \O(\g_n)$ has finite dimensional $\h$-weight spaces and the formal character $\ch M$ can be expressed uniquely as a non-negative integral linear combinations of characters of simple modules. We call the coefficients appearing in these sums {\it the composition multiplicities of $M \in \O(\g_n)$}. We explain this in more detail in Section~\ref{S:compmult}, and give a description (see Lemma~\ref{compmultalternatedefn}) of the $[M : L_\lambda]$ in the spirit of composition multiplicities for affine Lie algebras \cite[Proposition~9.7]{Kac}.\vspace{2pt}

We now state the main result of this paper, which follows directly from Corollary~\ref{C:reducingthelevel}
\begin{theorem}
\label{T:intromain}
Let $n > 0$, let $G$ be a connected reductive group and $\lambda, \mu \in \h_n^*$. The composition multiplicity $[M_\lambda : L_\mu]$ can be expressed via a precise formula \eqref{e:compmultreduction} in terms of composition multiplicities of simple modules inside Verma modules for a truncated current algebra $\l_{n-1}$ where $\l = \Lie(L)$ for some Levi subgroup $L \subseteq G$.
\end{theorem}

Theorem~\ref{T:intromain} suggests an inductive procedure for calculating the composition multiplicities $[M_\lambda : L_\mu]$ for the $n$th truncated current algebra in terms of the analogous composition multiplicities for $\g_0$. By the proof of the Kazhdan--Lusztig conjecture by Beilinson--Bernstein and Brylinski--Kashiwara (see \cite{HTT}) we know that for $n = 0$ the composition multiplicities $[M_\lambda : L_\mu]$ are given by the evaluation at 1 of certain Kazhdan--Lusztig polynomials. Hence for $n > 0$ the values $[M_\lambda : L_\mu]$ can be described by formulas involving non-negative integral linear combinations of these Kazhdan--Lusztig polynomials at 1. It is interesting to wonder whether our formulas have a natural geometric interpretation on the flag variety. \vspace{2pt}

We remark that the methods of this article generalise those of \cite{Ch}. In the rest of the introduction we highlight some of our other key results, and describe the structure of the paper.\vspace{2pt}

In Section~\ref{S:prelims} we introduce the basic notation which will be used throughout the paper. We then recall the work of Ra{\"i}s--Tauvel \cite{RT} on $S(\g_n)^{\g_n}$ and use it to describe the generators of the centre of the enveloping algebra $U(\g_n)$.\vspace{2pt}

In Section~\ref{S:catO} we describe all of the elementary properties of our category $\O(\g_n)$. We also explain that the category can be decomposed into a direct sum of abelian subcategories which are parameterised by the generalised eigenspaces of $\h\otimes t^i$ for $i =1,...,n$. The generalised eigenvalues are parametrised by $(\h_n^{\ge 1})^*$, and for $\mu \in (\h_n^{\ge 1})^*$ we call the subcategory $\O^{(\mu)}(\g_n)$ a {\it Jordan block of} $\O(\g_n)$. \vspace{2pt}

In Section~\ref{S:parabolicinduction} we introduce one of our main tools for simplifying the the study of Jordan blocks of $\O(\g_n)$. Let $\mu \in (\h_n^{\ge 1})^*$ and let $\mu_n := \mu|_{\h\otimes t^n}$, which we identify with an element of $\h^*$ in the obvious fashion. The centraliser $\l := \g^{\mu_n}$ is a Levi subalgebra containing $\h$, and we suppose that it is the Levi factor of a standard parabolic $\p\subseteq \g$ such that $\r = \operatorname{Rad}(\p) \subseteq \n^+$. In this case we can parabolically induce modules from $\O^{(\mu)}(\l_n)$ to $\O^{(\mu)}(\g_n)$. Similarly we have a functor of $\r_n$-invariants in the opposite direction.\vspace{2pt}

A precise statement of the following result is given in Theorem~\ref{maintheorem}.
\begin{theorem}
\label{T:introparabolic}
If $\l = \g^{\mu_n}$ is a standard Levi subalgebra, parabolic induction and $\r_n$-invariants are quasi-inverse equivalences between $\O^{(\mu)}(\l_n)$ and $\O^{(\mu)}(\g_n)$.
\end{theorem}
In the case where $\g^{\mu_n}$ is a standard Levi subalgebra, the theorem allows us to reduce the study of $\O^{(\mu)}(\g_n)$ to the case where $\mu_n$ is supported on the centre of $\g$. Another simple reduction (Lemma~\ref{L:supportonthecentreequiv}) impels us to focus on the case $\mu_n = 0$.

Theorem~\ref{T:introparabolic} was inspired by a result of Friedlander--Parshall in modular representation theory of Lie algebras \cite[Theorem~3.2]{FP}. Using the fact that modules in $\O(\g_n)$ admit finite filtrations with highest weight sections (Lemma~\ref{finitefiltration}) the proof quickly reduces to showing that $(\bullet)^{\r_n}$ is exact. This is the hardest part of the proof, and requires a careful study of central characters of highest weight modules (Theorem~\ref{vermaexttheorem}), which ultimately depends on our description of the centre of $U(\g_n)$ given in Section~\ref{ss:centreofenvelopingalg}. In particular, in comparison to the Takiff case treated in \cite[\S4]{Ch} the description of the central elements given by \cite{RT} is more difficult to work with.

In Section~\ref{S:twisting} we remove the requirement in Theorem~\ref{T:introparabolic} that $\l$ is a standard Levi subalgebra; this is achieved with the use of twisting functors. These were first introduced by Arkhipov \cite{Ar} and were applied to category $\O$ for $\g$ by Andersen--Stroppel \cite{AS}. They were used in the context of Takiff Lie algebras by the first author in \cite{Ch}, whilst proving a similar reduction to that of the present paper. The main result of Section~\ref{S:twisting} is Theorem~\ref{maintwistingtheorem}, which implies, together with Proposition \ref{P:weylgroupstandardparabolics}, the following:
\begin{theorem}
\label{T:introtwisting}
Every Jordan block of $\O(\g_n)$ is equivalent to a Jordan block $\O^{(\mu)}(\g_n)$ such that $\g^{\mu_n}$ is a standard Levi subalgebra.
\end{theorem}
Together with Theorem~\ref{T:introparabolic} this allows us to reduce the study of Jordan blocks $\O^{(\mu)}(\g_n)$ to the case where $\mu_n = 0$. The proof of Theorem~\ref{T:introtwisting} is similar to \cite[Theorem~5.7]{Ch} however several of the proofs, especially the proofs of Lemmas \ref{homconstruction} and \ref{maintwistingpartc}, are significantly more complicated. As such we provide detailed arguments for all of the proofs which are different, and refer the reader to the appropriate part of {\it op. cit.} otherwise. \vspace{2pt}

In Section~\ref{S:compmult} we address the problem of defining composition multiplicities and computing multiplicities of simples in Vermas. We begin by explaining how these numbers are defined and interpreting them in terms of composition series, which is all very similar to \cite[\textsection 6.1]{Ch}, except that the structure of the filtrations considered in Lemma \ref{compmultalternatedefn} is more complicated than the analogus filtrations in the Takiff case. Finally we focus on the blocks $\O^{(\mu)}(\g_n)$ with $\mu_n = 0$. The main result of the section is Corollary~\ref{C:reducingthelevel} which is a precise version of Theorem~\ref{T:intromain}.

\subsection*{Acknowledgements} 
Both authors would like to thank Simon Goodwin for many useful conversations. The first author is grateful to the EPSRC for studentship funding, and the research of the second author is supported by a UKRI FLF, grant numbers MR/S032657/1, MR/S032657/2, MR/S032657/3.

\section{Preliminaries}
\label{S:prelims}
In this paper all vector spaces, algebras and algebraic varieties will be defined over $\C$. Unadorned tensor products are taken over $\C$.

\subsection{Reductive Lie algebras and truncated currents}

From henceforth we fix a reductive algebraic group $G$ of rank $r$, with Lie algebra $\g$, and a choice of maximal torus $\h \subseteq \g$. The Weyl group $N_G(\h) / C_G(\h)$ will be denoted $W$. Let $\Phi \subseteq \h^*$ be the root system of $\g$ and let $\Delta$ be a choice of simple roots for $\Phi$, which give a set of positive roots $\Phi^+ = \Phi \cap \sum_{\alpha \in \Delta} \Z_{\ge 0} \alpha$. For each root $\alpha \in \Phi$ we have a one dimensional root space $\g_\alpha$ and for positive roots $\alpha$ we fix a choice of triple $(e_\alpha, h_\alpha, f_\alpha)$ such that $e_\alpha \in \g_\alpha$, $f_\alpha \in \g_{-\alpha}$, and $h_\alpha := [e_\alpha, f_\alpha]$, satisfying $\alpha(h_\alpha) = 2$.

These data give us a triangular decomposition
\begin{eqnarray}
\g &=& \n^- \oplus \h \oplus \n^+,
\end{eqnarray}
where $\n^\pm :=  \bigoplus_{\alpha \in \Phi^+} \g_{\pm\alpha}$. We also write $\mathfrak{b} = \h \oplus \n^+$ for the corresponding Borel subalgebra of $\g$. 

For $n \geq 1$ we consider the {\it truncated current algebra} $\g_n := \g \otimes \C[t]/(t^{n+1})$. For any subalgebra $\mathfrak{s} \subseteq \g$, have a natural embedding $\mathfrak{s}_n \subseteq \g_n$ of truncated currents. We make the notation $\mathfrak{s}_n^i := \mathfrak{s}\otimes t^i \subseteq \g_n$. This gives a grading $\mathfrak{s}_n = \bigoplus_{i=0}^n \mathfrak{s}_n^i$. We write $\mathfrak{s}_n^{\ge m}$ for the sum of the graded pieces of degree $m,m+1,...,n$.

 For $x \in \g$ and $\alpha \in \Phi^+$ then we make the notation
\begin{eqnarray}
\label{e:introduceehf}
\begin{array}{rcl}
x_i &:=& x \otimes t^i, \\ 
e_{\alpha, i} &:=& e_\alpha \otimes t^i,\\
f_{\alpha, i} &:=& f_{\alpha}  \otimes t^i,\\
h_{\alpha, i}&:=& h_\alpha \otimes t^i
\end{array}
\end{eqnarray}
We will often need to consider linear functions of $\g_n$ and $\h_n$. If $\lambda \in \h_n^*$ then we write
\begin{eqnarray}
\label{e:introducelambdarestricted}
\begin{array}{rcl}
\lambda_i & :=&  \lambda|_{\h_n^i}, \\
\lambda_{\ge i} &:=& \lambda|_{\h_n^{\ge i}},
\end{array}
\end{eqnarray}
 for $i=0,...,n$, and we often view $\lambda_i$ as an element of $\h^*$ via the obvious identification $\h = \h_n^i$. If $\lambda \in \h_n^*$ then we can also view it as an element of $\g_n^*$ by setting $\lambda(\n_n^{\pm}) = 0$.

\subsection{The symmetric invariants and the centre of the enveloping algebra}
\label{ss:centreofenvelopingalg}

We write $S(\g_n)$ and $U(\g_n)$ for the symmetric algebra and the universal enveloping algebra of $\g_n$, respectively. The algebra $U(\g_n)$ is equipped with the Poincar{\'e}--Birkhoff--Witt (PBW) filtration, and the associated graded algebra is $\gr U(\g_n) = S(\g_n)$ the symmetric algebra on $\g_n$. In the present section we describe the centre $Z(\g_n)$ of the enveloping algebra and its semi-classical limit.

The adjoint representation of $\g_n$ extends uniquely to an action of $\g_n$ on both $S(\g_n)$ and $U(\g_n)$ by derivations \cite[Proposition~2.4.9]{D}. The centre $Z(\g_n)$ is equal to $U(\g_n)^{\g_n}$ and the identification $\gr U(\g_n) = S(\g_n)$ is $\g_n$-equivariant. The invariant algebra $S(\g_n)^{\g_n}$ was first described by Ra{\"i}s--Tauvel \cite{RT}, as we now recall.


We define a series of (vector space) endomorphisms $\partial^{(0)},..., \partial^{(n)}$ of $S(\g_n)$ by putting $\partial^{(0)}$ equal to the identity map, and then inductively defining
\begin{eqnarray}
\label{e:HSder1}
& & \partial^{(k)}(x_j) = \left\{\begin{array}{cl} \binom{j}{k} x_{j-k} & \text{ if } j \ge k,\\ 0 & \text{ otherwise}, \end{array} \right.\\
\label{e:HSder2}
& & \partial^{(k)}(fg) = \sum_{i+j = k} \partial^{(i)}(f) \partial^{(j)}(g).
\end{eqnarray}
for $x \in \g$ and $f,g \in S(\g_n)$. These endomorphisms will be used to construct the basic invariant generators introduced in {\it op. cit.} We remark that the family of operators $(\partial^{(0)},...,\partial^{(n)})$ which we define here and an instance of higher order derivation \cite{SH}, but we shall not use this formalism.

Recall Chevalley's restriction theorem, which states that $S(\g)^G = S(\g)^\g \cong S(\h)^W$. Furthermore the Chevalley--Sheppard--Todd theorem implies that $S(\g)^\g$ is a graded polynomial algebra generated by $\rank(\g)$ homogeneous elements. Write $p_1,...,p_r$ for a choice of such elements.

The adjoint representation of $\g_n$ stabilises $\g_n^n$ and the action factors through the map $\g_n \onto \g_n / \g_n^{\ge 1} = \g_0 = \g$. Therefore we have a natural inclusion $S(\g)^\g \hookrightarrow S(\g_n^n)^{\g_n}$. Abusing notation we view $p_1,...,p_r$ as elements of $S(\g_n)^{\g_n}$.
\begin{theorem} \cite[\textsection 3]{RT}
\label{RT}
The invariant algebra $S(\g_n)^{\g_n}$ is a polynomial ring generated by $(n+1)r$ elements
\begin{eqnarray}
\label{e:symmetricinvariants}
\{\partial^{(k)} p_j \mid j=1,...,r, \ k = 0,...,n\}.
\end{eqnarray}
\end{theorem}

Let $d^{(k)} : S(\g_n^n) \to S(\g_n)$ be the partial derivative
$\sum_x x_{n-k} \frac{d}{dx_n}$
where the sum is taken over a basis for $\g$. Then it follows straight from the definitions \eqref{e:HSder1}, \eqref{e:HSder2} that there exist elements $q_j^{(k)} \in S(\g_n^{\ge n-k+1})$ such that
\begin{eqnarray}
\label{e:propertiesofdp}
\partial^{(k)}p_j = d^{(k)}p_j + q_j^{(k)} \text{ for } j=1,...,r, \ k = 0,...,n.
\end{eqnarray}
We refer the reader to \cite[Lemma~3.2(ii)]{RT} for the proof of \eqref{e:propertiesofdp}. The next observation follows directly from the definitions, and we record it as a lemma for later use.
\begin{lemma}
\label{L:dandhinvariants}
The map $d^{(k)}$ sends $\ad(\h)$-invariants to $\ad(\h)$-invariants. $\hfill\qed$
\end{lemma}

Finally we give a description of the centre $Z(\g_n)$ of the enveloping algebra. There is an isomorphism of $\g_n$-modules $\omega : S(\g_n) \to U(\g_n)$ called the symmetrisation map \cite[\textsection 2.4]{D}. It is defined by the rule
\begin{eqnarray}
\label{e:omegadefn}
\omega : x^1\cdots x^m \mapsto \frac{1}{m!} \sum_{\sigma\in \Sf_m} x^{\sigma 1} \cdots x^{\sigma m}
\end{eqnarray}
where $x^1,...,x^m \in \g_n$ are any elements, so that $x^1\cdots x^m \in S(\g_n)$ is a monomial of degree $m$, and $\Sf_m$ denotes the symmetric group on $m$ letters. If $\gr : U(\g_n) \to S(\g_n)$ is the (non-linear) map defined by taking the top degree component with respect to the Poincar{\'e}--Birkoff--Witt filtration, then
\begin{eqnarray}
\label{e:omegaleftinverse}
\gr \circ \ \omega \text{ is equal to the identity mapping on each graded piece of } S(\g_n).
\end{eqnarray}

Since the isomorphism $\omega : S(\g_n)^{\g_n} \isoto U(\g_n)^{\g_n} = Z(\g_n)$ respects filtrations we can describe the $Z(\g_n)$.
\begin{corollary}
$Z(\g_n)$ is a polynomial algebra generated by $(n+1)r$ elements
\begin{eqnarray}
\label{e:centralgenerators}
\{\omega(\partial^{(k)} p_j) \mid j=1,...,r, \ k = 0,...,n\}.
\end{eqnarray}
\end{corollary}
\begin{proof}
Let $Z \subseteq Z(\g_n)$ be the subalgebra generated by the elements \eqref{e:centralgenerators}. Since the inclusion $\gr Z \subseteq \gr Z(\g_n)$ is an equality it follows that the inclusion $Z \subseteq Z(\g_n)$ is also an equality. If these elements admit a non-trivial algebraic relation then taking $\gr$ and using \eqref{e:omegaleftinverse} we see that the elements \eqref{e:symmetricinvariants} would admit a non-trivial relation. By Theorem~\ref{RT} we see that  \eqref{e:centralgenerators} are algebraically independent.
\end{proof}

\section{Category $\calO$ for Truncated Current Lie algebras}
\label{S:catO}
\subsection{Definition and first properties of $\calO$}

We begin by stating the definition of category $\calO$ for $\g_n$.

\begin{definition}
\label{catOdef}
The category $\calO(\g_n)$ is the full subcategory of $U(\g_n)$-Mod with objects $M$ satisfying the following:

\begin{enumerate}
\setlength{\itemsep}{4pt}
    \item[$(\calO1)$] $M$ is finitely generated.
    \item[$(\calO2)$] $\h_n^0$ acts semisimply on $M$.
    \item[$(\calO3)$] $\h_n^{\ge 1}\oplus \n_n^+$ acts locally finitely on $M$.
\end{enumerate}
Note that $\O(\g_0)$ is nothing but the Bernstein--Gelfand--Gelfand (BGG) category $\O$ for $\g_0 = \g$. 
\end{definition}

We refer the reader to \cite{H} for a fairly comprehensive introduction to the algebraic study of $\O$, and to \cite{HTT} for a discussion of the relationship between $\O$ and $\mathcal{D}$-modules on the flag variety.

It is not hard to see that $\calO(\g_n)$ is closed under submodules, quotients, and finite direct sums. Furthermore by ($\calO1)$ every $M \in \calO(\g_n)$ is noetherian, since $U(\g_n)$ is noetherian.

Let $M \in \calO(\g_n)$ and $\lambda \in \h^*$. Define the {\it weight space of weight $\lambda$} by
\begin{eqnarray}
M^\lambda = \{ v\in M \mid h_0 v = \lambda(h) v\text{ for all } h\in \h\}.
\end{eqnarray}
By ($\O2$) we have $M = \bigoplus_{\lambda \in \h^*} M^\lambda$ and this is a module grading of $M$, if we equip $U(\g_n)$ with its natural grading by the root lattice. The elements of $M^\lambda$ are called \emph{weight vectors of weight $\lambda \in \h^*$}. If $m \in M^\lambda$ is a weight vector satisfying $\n_n^+ \cdot m = 0$ then we say that $m$ is a {\it maximal vector of weight $\lambda$}.

Now let $\lambda \in \h_n^*$ and recall that $\lambda_i := \lambda|_{\h_n^i}$. We say that $m \in M^{\lambda_0}$ is a \emph{highest weight vector of weight $\lambda$} if $m$ is maximal of weight $\lambda$ and 
\begin{eqnarray}
h \cdot m = \lambda(h)m \text{ for all } h\in \h_n.
\end{eqnarray}

The following basic properties of weight spaces of $M \in \O(\g_n)$ can be proven using the same argument as in BGG category $\O$, see \cite[\S1.1]{H}:
\begin{eqnarray}
\label{fdweightspaces}
\begin{array}{l}
\dim(M^\lambda) < \infty \text{ for all } \lambda \in \h^*,\vspace{6pt}\\
\{\lambda \in \h^* : M^{\lambda} \neq 0\} \subseteq \bigcup_{\lambda \in I} (\lambda - \Z_{\geq 0}\Phi^+) \text{ for some finite subset }I \subseteq \h^*.
\end{array}
\end{eqnarray}

If $m\in M \in \O(\g_n)$ is a maximal vector then we can find a highest weight vector in $U(\h_n) v$ thanks to $(\O3)$. This proves the following result.
\begin{lemma} \cite[Corollary~3.3]{Ch}
\label{maxhwcorollary}
Suppose that $M \in \calO$ admits a nonzero maximal vector of weight $\lambda \in \h^*$ in $M$. Then  $M$ admits a nonzero highest weight vector of weight $\mu$ for some $\mu \in \h_n^*$ satisfying $\mu_0 = \lambda$. $\hfill\qed$
\end{lemma}

\subsection{Highest weight modules}
\label{ss:hwmodules}
We say $M$ is a {\it highest weight module of weight $\lambda\in \h_n^*$} if $M$ is generated by a highest weight vector of weight $\lambda$. The following result on highest weight filtrations is analogous to the situation for BGG category $\O$, see \cite[Corollary 1.2]{H}. The proof is essentially the same as \cite[Lemma~3.4]{Ch}.
\begin{lemma}
\label{finitefiltration}
Let $M \in \calO(\g_n)$. Then $M$ has a finite filtration  $0 = M_0 \subseteq M_1 \subseteq \cdots \subseteq M_k = M$
such that each $M_{i+1}/M_i$ is a highest weight module. We call such a filtration a {\it highest weight filtration}. $\hfill\qed$
\end{lemma}

For $\lambda\in \h_n^*$ we define the {\it Verma module of weight $\lambda$} via
\begin{eqnarray}
M_{\lambda} := U(\g_n) \otimes_{U(\mathfrak{b}_n)} \C_{\lambda},
\end{eqnarray}
where $\C_{\lambda}$ is the one dimensional $U(\mathfrak{b}_n)$-module upon which $\h_n$ acts via $\lambda$, and $\n_n$ acts by 0. The Verma modules are the universal highest weight modules, in the sense that every highest weight module is a quotient of a Verma module.

They enjoy the following nice properties, generalising the classical case $n = 0$, see \cite[\textsection 1]{H}: 
\begin{enumerate}
\setlength{\itemsep}{4pt}
\item $\dim M_{\lambda}^{\lambda_0}= 1$, and hence $\dim M^{\lambda_0} = 1$ for every highest weight module of weight $\lambda\in \h_n^*$.
\item Every highest weight module $M$ admits a central character: for every $\lambda \in \h_n^*$ there is a homomorphism $\chi_\lambda : Z(\g_n) \to \C$ such that $z\cdot m = \chi_\lambda(z) m$ for all $z\in Z(\g_n)$ and all $m \in M$ where $M$ is a highest weight module, of weight $\lambda$.
\item $M_\lambda$ admits a unique maximal submodule and a unique simple quotient, which we denote $L_\lambda$.
\item Every simple object in $\O(\g_n)$ is isomorphic to precisely one of these simple modules (by Lemma~\ref{finitefiltration}). Thus the modules
$$\{L_{\lambda} \mid \lambda \in \h_n^*\}$$
give a complete set of representatives for the isomorphism classes of simple modules in $\O(\g_n)$.
\end{enumerate}

\subsection{Jordan decomposition for $\calO(\g_n)$}
The standard approach to studying modules in BGG category $\O$ is to consider modules with a fixed central character for $U(\g)$. This refinement is also useful in our more general setting (see Theorem~\ref{vermaexttheorem}), however as a first approximation we decompose $\O(\g_n)$ in terms of generalised eigenvalues for $\h_n^{\ge 1}$. 

Fix $M \in \O(\g_n)$ and $\mu \in (\h_n^{\ge 1})^*$. We define the {\it generalised eigenspace of eigenvalue $\mu$} via
$$M^{(\mu)} = \{m \in M \mid (h -\mu(h))^k m = 0 \text{ for all } k \gg 0, \ h \in \h_n^{\ge 1}\}.$$

The following result is a slight generalisation of \cite[Lemma~3.7]{Ch}, and we supply a sketch of the proof for the reader's convenience.
\begin{lemma}
\label{generalisedeigenvalues}
Every $M \in \calO(\g_n)$ admits a direct sum decomposition of $\g_n$-modules
\begin{eqnarray}
\label{e:Mdecomposes}
M = \bigoplus_{\mu \in (\h^*)^n} M^{(\mu)}.
\end{eqnarray}
\end{lemma}

\begin{proof}
Since $\h_n$ preserves the weight spaces of $M$, which are finite dimensional, it follows that each $M^\lambda$ decomposes into generalised eigenspaces for $\h_n^{\ge 1}$. Therefore $M$ admits a decomposition \eqref{e:Mdecomposes} and it suffices to show that each summand is a $\g_n$-module. This follows by a direct calculation, using the fact that $\g$ admits an eigenbasis for $\h$ (root space decomposition) and $\h_n^{\ge 1}$ acts nilpotently on $\g_n$.
\end{proof}

Now we define the {\it Jordan block of $\O(\g_n)$ of weight $\mu \in (\h_n^{\ge 1})^*$} to be the full subcategory $\calO(\g_n)$ whose objects are the modules $M$ such that $M = M^{(\mu)}$. We then have the following {\it Jordan decomposition}
\begin{eqnarray}
\label{calOdecomposition}
\calO(\g_n) = \bigoplus_{\mu \in (\h_n^{\ge 1})^*} \O^{(\mu)}(\g_n).
\end{eqnarray}

\begin{remark}
\label{R:Vermasinblocks}
It is not hard to see that if $\lambda \in \h_n^*$ and $\mu = \lambda|_{\h_n^{\ge 1}}$ then both $M_{\lambda}$ and $L_\lambda$ lie in $\O^{(\mu)}(\g_n)$ (see also \cite[Lemma~3.9]{Ch}).  Combining \eqref{calOdecomposition} with the fact that Verma modules have unique maximal submodules, and are therefore indecomposable, it follows that $L_{\lambda}$ cannot occur as a subquotient of $M_{\nu}$ unless $\lambda_{\ge 1} = \nu_{\ge 1}$

\end{remark}

Let $\g_n \to \g_{n-1}$ be the natural quotient map with kernel $\g_n^n$, and consider the pull-back functor
\begin{eqnarray}
\label{e:isafunctor}
p : \O(\g_{n-1}) \longrightarrow \O(\g_n).
\end{eqnarray}

\begin{lemma}
\label{simple0modules}
Let $\lambda \in \h_{n-1}^*$ and define $\nu \in \h_n^*$ by $\nu(h_i) = \lambda(h_i)$ for $i=0,...,n-1$ and $\nu(h_n) = 0$ for all $h\in \h$. Then
$p(L_\lambda)\cong L_{\nu}$ as $\g_n$-modules.
\end{lemma}
\begin{proof}
Certainly $p(L_\lambda)$ is a simple highest weight module of highest weight $\nu$, and the proof follows.
\end{proof}

Now we state and prove some easy equivalences between Jordan blocks of $\O(\g_n)$ which arise by tensoring with one dimensional $\g_n$-modules. We write $\g' = [\g,\g]$ for the derived subalgebra. We note that $(\g_n)' = (\g')_n$, and so we may use the notation $\g_n'$ unambiguously.

For $\lambda\in \h_n^*$ we recall the notation $\lambda_{\ge 1}:=\lambda|_{\h_n^{\ge 1}}$. Any such $\lambda$ can be extended to an element of $\g_n^*$ via $\lambda(\n_n^{\pm}) = 0$, and we may abuse notation by identifying $\h_n^*$ with a subspace of $\g_n^*$. For $\lambda \in \Ann_{\g^*_n} (\g'_n)$ let $\C_{\lambda}$ be the one dimensional $\g_n$-module afforded by $\lambda$. 
 
\begin{lemma}
\label{L:supportonthecentreequiv}
Suppose that $\lambda, \nu \in \h_n^*$ such that $\lambda|_{\g_n'} = \nu|_{\g_n'}$. Then $(\bullet) \otimes_{U(\g_n)} \C_{\lambda - \nu}$ and $(\bullet) \otimes_{U(\g_n)} \C_{\nu - \lambda}$ are quasi-inverse equivalences between $\O^{(\lambda_{\ge 1})}(\g_n)$ and $\O^{(\nu_{\ge 1})}(\g_n)$.
\end{lemma}
\begin{proof}
Since $\pm(\lambda - \nu)$ vanishes on $\g_n'$ it defines a one dimensional representation of $\g_n$, and the named functors are quasi-inverse autoequivalences of $\g_n\lmod$. To complete the proof it suffices to observe check that $(\bullet) \otimes_{U(\g_n)} \C_{\lambda - \nu}$ sends $\O^{(\lambda_{\ge 1})}(\g_n)$ to $\O^{(\nu_{\ge 1})}(\g_n)$, which follows directly from the definitions.
\end{proof}

\section{Parabolic Induction}
\label{S:parabolicinduction}

In this section we prove Theorem~\ref{T:introparabolic}, which allows us to relate the category $\O^{(\mu)}(\g_n)$ with a Jordan block of $\O(\l_n)$ for a Levi subalgebra $\l$. Recall that if $\nu \in \h^*$ then we extend $\nu$ to an element of $\g^*$ via $\nu(\n^\pm) = 0$, and write $\g^\nu$ for the coadjoint centraliser. 

\begin{theorem}
\label{maintheorem}
Let $\lambda \in \h_n^*$. Suppose that the centraliser $\l = \g^{\lambda_n}$ is in standard Levi form and let $\p := \l + \n^+ =  \l \oplus \r$ be a parabolic subalgebra with Levi factor $\l$ and nilradical $\r$. Write $\mu = \lambda_{\ge 1}$ (notation \eqref{e:introducelambdarestricted}). The categories $\calO^{(\mu)}(\l_n)$ and $\calO^{(\mu)}(\g_n)$ are equivalent. The quasi-inverse functors inducing the equivalence are parabolic induction and  $\r_n$-invariants:
\begin{eqnarray*}
\begin{array}{rcccl}
\Ind & : & \calO^{(\mu)} (\l_n) &\longrightarrow & \calO^{(\mu)} (\g_n) \vspace{3pt} \\
& & M &\longmapsto & U(\g_n) \otimes_{U(\p_n)} M\vspace{6pt} \\
(\bullet)^{\r_n} & : & \calO^{(\mu)} (\g_n) &\longrightarrow  &
\calO^{(\mu)} (\l_n) \vspace{3pt} \\
& & M &\longmapsto & M^{\mathfrak{r}_n}.
\end{array}
\end{eqnarray*}
\end{theorem}

Theorem~\ref{maintheorem} is inspired by a category equivalence in modular representation theory due to Friedlander and Parshall \cite[Theorem 2.1]{FP}. The case $n = 1$ is due to the first author \cite{Ch} and our method here is a generalisation of {\it loc. cit.}

We observe that $\Ind$ is left adjoint to $(\bullet)^{\r_n}$ since we have inverse isomorphisms $$\Hom_{\l_n}(M, N^{\mathfrak{r}_n}) \underset{\eta}{\overset{\theta}{\myrightleftarrows{\rule{1cm}{0cm}}}} \Hom_{\g_n}(\Ind M, N)$$ 
given by $\theta(f)(u \otimes m) = u \cdot f(m)$ and $\eta(g)(m) = g(1 \otimes m)$ for $u\in U(\g_n)$ and $m \in M$.

Let $\Id_\mathcal{C}$ denote the identity endofunctor of a category $\mathcal{C}$. In order to show that the adjoint functors $\Ind$ and $(\bullet)^{\r_n}$ are equivalences we consider the unit and counit of the adjunction, see \cite[II.7]{HS} for example. The unit is the natural transformation $\psi : \Id_{\O(\g_n)} \to (\bullet)^{\r_n} \circ \Ind$ which is obtained by applying $\eta$ to the identity mapping $N^{\r_n} \to N^{\r_n}$, whilst the counit is the natural transformation $\varphi : \Ind \circ (\bullet)^{\r_n} \to \Id_{\O(\l_n)}$ obtained by applying $\theta$ to the identity mapping $\Ind M \to \Ind M$. In particular we have
\begin{eqnarray}
\label{e:adjunctionmorphisms}
\begin{array}{rcl}
\psi_M &:& M \longrightarrow (\Ind M)^{\mathfrak{r}_n}\\
& & m \longmapsto 1\otimes m,\vspace{4pt}\\
\varphi_N & : & \Ind (N^{\mathfrak{r}_n}) \longrightarrow N\\
& & u\otimes m \longmapsto u\cdot m.
\end{array}
\end{eqnarray} 
In order to complete the proof of Theorem~\ref{maintheorem} it suffices to show that $\psi$ and $\varphi$ are both natural equivalences. The proof, which is given is Section~\ref{ss:exactinvariants}, depends heavily on the exactness of $(\bullet)^{\r_n}$, which will occupy the majority of Sections~\ref{ss:extandcc} and \ref{ss:exactinvariants}.

\subsection{Central characters}
\label{ss:extandcc}

If two $\g_n$-modules admit different infinitesimal central characters then there can be no extensions between them. The main step in proving exactness of $(\bullet)^{\r_n}$ is the following result which leads to a vanishing criterion for extensions. For $\nu \in \h^*$ we make the notation $$\Phi_\nu := \{\alpha \in \Phi : \nu(h_\alpha) = 0\}.$$

For $\lambda \in \h_n^*$ recall the notation $\chi_\lambda$ for central characters, introduced in Section~\ref{ss:hwmodules}, (2).

\begin{theorem}
\label{vermaexttheorem}
Let $\lambda, \lambda' \in \h^*_n$ such that $\lambda_{\ge 1} = \lambda'_{\ge 1}$ and $\g^{\lambda_n}$ is in standard Levi form. Then $\chi_\lambda = \chi_{\lambda'}$ if and only if $\lambda_0 - \lambda'_0 \in \C\Phi_{\lambda_n}$.
\end{theorem}

\begin{corollary}
\label{C:vermaextensions}
If $M \in \calO^{(\mu)}(\g_n)$ is indecomposable and $\g^{\mu_n}$ is in standard Levi form, then there is a unique coset $\Xi_M \in \h^*/ \C\Phi_{\mu_n}$ such that if $N$ is a highest weight subquotient of $M$ of weight $\lambda \in \h_n^*$ then $\lambda_{\ge 1} = \mu$ and $\lambda_0 + \C\Phi_{\mu_n} = \Xi_M$.
\end{corollary}

Before we proceed, we explain how to deduce the Corollary from Theorem~\ref{vermaexttheorem}. Using the Jordan decomposition from Lemma~\ref{generalisedeigenvalues} along with Remark~\ref{R:Vermasinblocks} we see that all subquotients of an indecomposable module lie in the same Jordan block $\O^{(\mu)}(\g_n)$, which confirms that $\lambda_{\ge 1} = \mu$. Since $M$ has finite dimensional weight spaces \eqref{fdweightspaces} and is indecomposable it admits a generalised central character, i.e. there is a unique maximal ideal $\m$ of $Z(\g_n)$ such that $\m^k M = 0$ for $k \gg 0$. Thus all of the highest weight subquotients have the same central character (\textsection \ref{ss:hwmodules}, (2)). Now Corollary~\ref{C:vermaextensions} follows from Theorem~\ref{vermaexttheorem}.

\begin{remark}
Later we shall see that Theorem~\ref{vermaexttheorem} and Corollary~\ref{C:vermaextensions} hold without the standard Levi type hypothesis. This follows from Theorem~\ref{T:introtwisting}. However since the latter theorem relies on Theorem~\ref{vermaexttheorem} we retain this hypothesis to keep these dependences clear.
\end{remark}

We now proceed to prove Theorem~\ref{vermaexttheorem}. Recall from Section~\ref{ss:centreofenvelopingalg} that the symmetric invariants $S(\g)^\g$ are generated by algebraically independent homogeneous elements $p_1,...,p_r$. Furthermore there is an embedding $S(\g)^\g \hookrightarrow S(\g_n)^{\g_n}$ such that $S(\g_n)^{\g_n}$ is generated by $r(n+1)$ elements $\partial^{(k)} p_j, \ j=1,...,r, \ k =0,...,n$. There is a linear map $\omega : S(\g_n)^{\g_n} \to Z(\g_n)$ and we will be especially interested in the central elements
$$z_i^{(j)} := \omega(\partial^{(j)} p_i) \ \ \text{ for } i = 1,...,r.$$
Now let $U(\g_n)^{\h}$ be the invariant subalgebra under the adjoint action of $\h$. Also let $U(\g_n) \n_n^+$ be the left ideal generated by $\n_n^+$. The intersection $U(\g_n) \n_n^+ \cap U(\g_n)^{\h}$ is actually an ideal of $U(\g_n)^\h$, and it is not hard to see that the quotient by this ideal is isomorphic to $U(\h_n)$. Following the observations of \cite[\textsection 1.7]{H} verbatim we see that $\chi_\lambda$ coincides with the composition
\begin{eqnarray}
\label{e:centcharalternative}
U(\g_n)^{\h} \to U(\h_n) = \C[\h_n^*] \overset{\ev_\lambda}{\longrightarrow} \C
\end{eqnarray}
where $\ev_\lambda$ denotes evaluation at $\lambda$. This allows us to consider $\chi_\lambda(p)$ for any $p \in U(\g_n)^{\h}$, not just $p \in Z(\g_n)$.

Now for $\mu \in (\h_n^{\ge1})^*$ we define two maps $\h^* \to \C^r$. For $\nu \in \h^*$ we use the notation $(\nu, \mu)$ to denote the element of $\h_n^*$ which restricts as $\nu$ on $\h_n^0$ and to $\mu$ on $\h_n^{\ge 1}$.  The first map
\begin{eqnarray*}
\xi_\mu(\nu) := (\chi_{(\nu,\mu)}\circ \omega \circ d^{(n)} p_1,..., \chi_{(\nu, \mu)} \circ \omega \circ d^{(n)} p_r) \in \C^r
\end{eqnarray*}
whilst the second is
\begin{eqnarray*}
\eta_\mu(\nu) := (\theta_{(\nu, \mu)}(p_1),...,\theta_{(\nu, \mu)}(p_r)) \in \C^r
\end{eqnarray*}
where $\theta_{(\nu, \mu)}$ denotes the composition
$$S(\g) = \C[\g^*] \to \C[\h^*] \to \C[\h_n^*] \to \C$$
and $\C[\g^*] \to \C[\h^*]$ is restriction across the triangular decomposition of $\g^*$, whilst the second map is the restriction of $d^{(n)}$ to $\C[\h^*]$ and the third map is evaluation $\ev_{(\nu, \mu)}$.
\begin{lemma}
\label{L:xiandetaswap}
$\xi_\mu(\nu) - \eta_\mu(\nu)$ depends only on $\mu$.
\end{lemma}
\begin{proof}
Note that we a Lie algebra embedding $\g_1 \hookrightarrow \g_n$ given by $x_0 \mapsto x_0$ and $x_1 \mapsto x_n$ for all $x\in \g$. We have $\xi_\mu(\nu) - \eta_\mu(\nu) \in S(\g_1) \subseteq S(\g_n)$ and this reduces the claim to the case $n = 1$. This was proven in \cite[Lemma~4.6(b)]{Ch}.
\end{proof}

\begin{lemma}
\label{L:characterlemma}
For $\lambda, \lambda' \in \h_n^*$ satisfying $\lambda_{\ge 1} = \lambda'_{\ge 1}$, 
 we have $\chi_\lambda = \chi_{\lambda'}$ if and only if $\chi_\lambda(z_i^{(n)}) = \chi_{\lambda'}(z_i^{(n)})$ for $i=1,...,r$.
\end{lemma}
\begin{proof}
Since the centre $Z(\g_n)$ is generated by the elements $\omega(\partial^{(k)} p_j)$ it follows that $\chi_\lambda = \chi_{\lambda'}$ if and only if the characters coincide on these elements. Thus if we can show that $\chi_\lambda(z_i^{(j)}) = \chi_{\lambda'}(z_i^{(j)})$ for all $j=0,...,n-1$ then the proof will be complete.

We have $p_i \in S(\g_n^n)$ and it follows from \eqref{e:HSder1} and \eqref{e:HSder2} that $\partial^{(j)}p_i \in S(\g_n^{\ge n-j})$. By the definition of $\omega$ \eqref{e:omegadefn} we have $z_i^{(j)} \in U(\g_n^{\ge n-j})$. Therefore $\chi_\lambda(z_i^{(j)})$ only depends on $\lambda_{\ge n-j}$. Since $\chi_\lambda$ is precisely the composition \eqref{e:centcharalternative}, we have reached the desired conclusion. This completes the proof.
\end{proof}

\begin{proofofextensiontheorem}
Let $\lambda, \lambda'$ satisfy the assumptions of the Theorem. Thanks to Lemma~\ref{L:characterlemma} we must show that $\chi_\lambda(z_i^{(n)}) = \chi_{\lambda'}(z_i^{(n)})$ for all $i$ is equivalent to the condition on $\lambda_0, \lambda_0'$.

Thanks to \eqref{e:propertiesofdp} we have $z_i^{(n)} = \omega(d^{(n)} p_j) + \omega(q_i^n)$. Since $q_i^n \in S(\g_n^{\ge 1})$ it follows from \eqref{e:omegadefn} that $\chi_\lambda(z_i^{(n)}) = \chi_{\lambda'}(z_i^{(n)})$ if and only if $\chi_\lambda\circ \omega \circ d^{(n)}(p_j) = \chi_{\lambda'} \circ\omega \circ d^{(n)}(p_j)$ for all $i=1,...,r$. We note that this second equality is well-defined because $d^{(n)}$ sends $\h$-invariants to $\h$-invariants (Lemma~\ref{L:dandhinvariants}), and $\omega$ is $\h_n$-equivariant, and $\chi_\lambda$ coincides with the composition \eqref{e:centcharalternative}.

Since $\mu := \lambda_{\ge 1} = \lambda'_{\ge 1}$ we can apply Lemma~\ref{L:xiandetaswap} to see that the central characters coincide if and only if $\eta_{\mu}(\lambda_0) = \eta_{\mu}(\lambda_0')$. Let $\pi : \h^* = \h^*/W \cong \C^r$ be the quotient map. It follows from the Chevalley restriction theorem \cite[3.1.37]{CGi} that we can write this in coordinates as $\pi(\nu) = (p_1|_{\h^*}(\nu),...,p_r|_{\h^*}(\nu))$. Now it is easy to see by a direct comparison of the two definitions that
$\eta_{\mu}(\lambda_0)$ coincides with the differential $d_{\lambda_n}\pi(\lambda_0)$ of the quotient map $\pi$.

By \cite[Lemma~4.6(e)]{Ch} we see that $\ker \eta_\mu = \ker d_{\lambda_n} \pi = \C\Phi_{\lambda_n}$. This implies that $\eta_{\mu}(\lambda_0) = \eta_{\mu}(\lambda_0')$ if and only if $\lambda_0 - \lambda_0' \in \C \Phi_{\mu_n}$ (this is where we use the fact that the centraliser in a standard Levi subalgebra). This completes the proof.
$\hfill\qed$
\end{proofofextensiontheorem}

\subsection{Exactness of the invariants functor}
\label{ss:exactinvariants}

In this section we fix a standard parabolic subalgebra $\p$ of $\g$, so that our fixed choice of positive root spaces are contained in $\p$. Pick a Levi decomposition $\p = \mathfrak{l} \oplus \mathfrak{r}$. The main result of this section is that $(\bullet)^{\r_n}$ is exact on $\O^{(\mu)}(\g_n)$ provided $\g^{\mu_n} = \l$. First we need the following Lemma.

\begin{proposition}
\label{highestweightmodinvariants}
Let $M \in \O^{(\mu)}(\g_n)$ be indecomposable module and let $\Xi_M \in \h^* / \C \Phi_{\mu_n}$ be the coset determined by $M$ in Corollary~\ref{C:vermaextensions}.  If $\l = \g^{\mu_n}$ is a standard Levi subalgebra then 
$$M^{\mathfrak{r}_n} = \bigoplus_{\nu \in \Xi_M} M^{\nu}.$$
\end{proposition}

\begin{proof}
Since $M$ is indecomposable and admits a highest weight subquotient of weight $\lambda$ we have $\lambda_n = \mu_n$. Since $\g^{\mu_n}$ is the Levi factor of a standard parabolic, we have $\mathfrak{r} = \vspan\{e_\alpha: \alpha \in \Phi^+ \backslash \Phi_{\lambda_n}\}$. By Lemma~\ref{finitefiltration} we have a finite filtration $0 \subseteq M_1 \subseteq \dots \subseteq M_k = M$ such that $M_i / M_{i-1}$ has highest weight $\lambda^{(i)} \in \h_n^*$. By \eqref{fdweightspaces} the weights of $M$ lie in the set
$\bigcup_{i=1}^k \{\lambda^{(i)}_0 - \sum_{\beta\in \Phi^+} k_\beta\beta \mid k_\beta \in \Z_{\ge 0}\}$.
Furthermore by Corollary~\ref{C:vermaextensions} there is an element $\Xi_M \in \h^*/ \C \Phi_{\mu_n}$ such that $\lambda_0^{(i)} + \C \Phi_{\mu_n} = \Xi_M$ for all $i$. It follows that the weights of $M$ actually lie in the set $\lambda^{(i)} + \C \Phi_{\mu_n} - \sum_{\beta \in \Phi^+} \Z_{\ge 0} \beta$, for any choice of $i$. 

 In particular if $\nu \in \h^*$ satisfies $\nu \in \lambda_0 + \Phi_{\lambda_n}$ then $\nu + \alpha$ does not lie in $\lambda^{(i)}_0 - \sum_{\beta \in \Phi^+} \Z_{\ge 0} \beta$ for any $\alpha \in \Phi^+ \setminus \Phi_{\lambda_n}$ and for any $i$. Therefore $\r_n \cdot M^{\nu} = 0$.

Conversely, suppose $v\in M^{\mathfrak{r}_n}$ is of weight $\nu\in \h^*$. Since $\h_n$ acts locally finitely and preserves weight spaces we can find a common eigenvector for $\h_n$ of weight $\nu$ in $U(\h_n)\cdot v$. Suppose the eigenvalue is $\lambda' \in \h_n^*$ (by assumption $\lambda'_0 = \nu$). Then a quotient of $M_{\lambda'}$ occurs as a submodule of $M$. All highest weight modules are indecomposable (they admit unique maximal submodules) and this forces the generalised central character of $M$ to be $\chi_{\lambda'}$. Now Theorem~\ref{vermaexttheorem} implies that $\nu \in \lambda_0' + \C\Phi_{\mu_n}$. By Corollary~\ref{C:vermaextensions} we see that $\lambda_0'$ and $\lambda_0$ lie in the same coset of $\h^*$ modulo $\C\Phi_{\mu_n}$, and so $\nu \in \lambda_0 + \C\Phi_{\mu_n}$.
\end{proof}


\begin{corollary}
\label{functorexactness}
Suppose that $\mu \in (\h_n^{\ge1})^*$ such that $\g^{\mu_n} = \mathfrak{l}$ is a standard Levi subalgebra. Then the functor $(\bullet)^{\mathfrak{r}_n} : \O^{(\mu)}(\g_n) \to \O^{(\mu)}(\g_n^{\mu_n})$ is exact.
\end{corollary}

\begin{proof}
It suffices to take a surjective morphism $M \to N$ in $\O^{(\mu)}(\g_n)$, and show that the restriction $M^{\mathfrak{r}_n} \rightarrow N^{\mathfrak{r}_n}$ is surjective. Without loss of generality we can assume that $M, N$ are indecomposable. The existence of a nonzero map $M \to N$ forces $\Xi_M = \Xi_N$. Since objects of $\O(\g_n)$ are $\h$-semisimple, it follows that $M \to N$ is surjective on $\h$-weight spaces, and now the result follows immediately from Proposition~\ref{highestweightmodinvariants}.
\end{proof}

\begin{proof}[Proof of Theorem \ref{maintheorem}]
We let $\varphi_M$ and $\psi_N$ be the adjunction morphisms from \eqref{e:adjunctionmorphisms}. 

Note that $\Ind$ is an exact functor because $U(\g_n)$ is free over $U(\p_n)$, thanks to the PBW theorem. Furthermore $(\bullet)^{\r_n}$ is exact by Corollary~\ref{functorexactness}. If we can check that $\varphi_M$ and $\psi_N$ are isomorphisms on highest weight modules, then a standard argument using the length of a highest weight filtration (Lemma~\ref{finitefiltration}) can be used to conclude that $\varphi_M$ and $\psi_N$ are isomorphisms for all $M \in \O^{(\mu)}(\g_n)$ and all $N \in \O^{(\mu)}(\g_n^{\mu_n})$

The map $\psi_M$ is an isomorphism for highest weight modules $M$, thanks to Proposition~\ref{highestweightmodinvariants}.

Now suppose that $N \in \O^{(\mu)}(\g_n)$ is a highest weight module with highest weight generator $v$. Since $\r_n \cdot v = 0$ it follows that $1\otimes v$ lies in the image of $\varphi_N$ and so $\varphi_N$ is surjective. To prove injectivity let $K = \ker(\varphi_N)$ and consider the short exact sequence $K \to \Ind(N^{\r_n}) \to N$. By Corollary~\ref{functorexactness} we have another short exact sequence $K^{\mathfrak{r}_n} \to \Ind(N^{\r_n})^{\r_n} \to N^{\r_n}$.

Now set $M = N^{\r_n}$, which is a highest weight $\g^{\mu_n}_n$-module generated by $v$. The map $\Ind(M)^{\r_n} \to M$ is $U(\g_n)$-equivariant map uniquely determined by $1\otimes m \mapsto m$. Therefore it is the left inverse of $\psi_M$, which we have already shown to be bijective. It follows that $K^{\r_n} = 0$, but since every nonzero object in $\O(\g_n)$ admits a nonzero highest weight vector, it follows that $K = 0$ and so $\varphi_N$ is an isomorphism for highest weight $N$.

This concludes the proof.
\end{proof}

\section{Twisting Functors}
\label{S:twisting}
\subsection{Definition of twisting functors}

Now we proceed to proving Theorem \ref{T:introtwisting}, extending the results of the first author \cite[\textsection 5]{Ch}. Throughout this section we fix a simple root $\alpha \in \Delta$ and make the notation $U := U(\g_n)$ for the sake of brevity.

Recall that a right Ore set $S$ in a non-commutative ring $R$ is a multiplicatively closed set of elements such that for all $r\in R$ and $s\in S$ there exists $r' \in R$ and $s' \in S$ such that $rs' = sr'$. A left Ore set is defined dually. In order for $R$ to admit a right ring of quotients with respect to $S$ it is necessary that $S$ is a right Ore set. When $R$ has no zero divisors the condition is also sufficient, and when $S$ is both a right and left Ore set the left and right fraction fields are isomorphic; see \cite[\textsection 2.1]{MR} for a survey of these facts.

Also recall the notation $e_{\alpha, i}, h_{\alpha, i}, h_{\alpha, i}$ from \eqref{e:introduceehf}. Let $F_\alpha$ be the multiplicative set generated by $\{f_{\alpha,0},...,f_{\alpha, n}\}$ and note that these elements commute amongst themselves. The proof of the following result is almost identical to \cite[Lemma~5.1]{Ch}.
\begin{lemma}
$F_\alpha$ is both a left and right Ore set in $U$. $\hfill\qed$
\end{lemma}

We now wish to explicitly construct the localisation $U$ with respect to $F_\alpha$. It certainly exists, by our previous remarks. Consider the $U$-algebra $U_{\alpha} := U[ F_\alpha^{-1}]$ which is freely generated by symbols $\{f_{\alpha, i}^{-1} \mid i=0,1,...,n\}$, subject to the relations $f_{\alpha, i} f_{\alpha,i}^{-1} = 1$.

We introduce some notation to describe a basis for $U_\alpha$. If 
\begin{eqnarray}
\label{e:introduceklm}
\begin{array}{l}
k : \Phi^+ \times \{0,1,...,n\} \to \Z_{\ge 0},\\
l : \Delta \times \{0,1,...,n\} \to \Z_{\ge 0},\\
m : \Phi^+ \setminus \{\alpha\} \times \{0,1,...,n\} \to \Z_{\ge 0}
\end{array}
\end{eqnarray}
are arbitrary maps of sets then we let
\begin{eqnarray}
\label{e:introducev}
v(k,l,m) = \Big(\prod_{i=0}^n \prod_{\beta \in \Phi^+} e_{\beta, i}^{k_{\beta, i}})
 \Big(\prod_{i=0}^n \prod_{\beta \in \Delta} h_{\beta, i}^{l_{\beta, i}})
  \Big(\prod_{i=0}^n \prod_{\beta \in \Phi^+\setminus \{\alpha\}} f_{\beta, i}^{m_{\beta, i}}\Big) \in U
\end{eqnarray}
where the product is taken with respect to some fixed choice of ordering on the basis of $\g$. More succinctly these elements are precisely the PBW monomials in $U$ which have no factor in $F_\alpha$.

\begin{lemma}
\label{L:basisLemma}
\begin{enumerate}
\item $U_\alpha$ is the (left and right) localisation of $U$ at the Ore set $F_\alpha$. 
\item A basis for $U_\alpha$ is given by the elements $f_{\alpha, 0}^{i_0} \cdots f_{\alpha, n}^{i_n} v(k,l,m)$ where $i_j \in \Z$  and $k,l,m$ are as in \eqref{e:introduceklm}.
\item A basis for $U_\alpha$ is given by the elements $v(k,l,m)f_{\alpha, 0}^{i_0} \cdots f_{\alpha, n}^{i_n}$ where $i_j \in \Z$  and $k,l,m$ are as in \eqref{e:introduceklm}.
\end{enumerate}

\end{lemma}
\begin{proof}
$U_\alpha$ satisfies the universal property of the localisation by construction, and so it is isomorphic to both the left and right localisation, thanks to \cite[Corollary~2.1.4]{MR}. This proves (1), and also implies that $U$ embeds inside $U_\alpha$, and $U_\alpha$ is an integral domain.

Note that $U_\alpha$ is spanned by unordered monomials in $\g_n$ and $F_\alpha^{-1}$. Using the left Ore condition we can rewrite any such monomial as a span of monomials of the form described in (2). Furthermore, if there is a linear dependence between the latter monomials, we can left multiply by appropriate elements of $F_\alpha$ to obtain a linear dependence between PBW monomials in $U$, which must be zero. This proves (2), and (3) follows by a symmetrical argument.
\end{proof}

We will need more precise relations between generators of $U_\alpha$.
\begin{lemma}
\label{relationslemma}
For any $i, j \in \Z$, the following relations hold in $U_\alpha$, for any $h \in \h$ and any $\beta \in \Phi^+ \setminus \{\alpha\}$:
\begin{align}
\label{relation1}
[e_{\alpha, i}, f_{\alpha, j}^{-1}] &= - f_{\alpha, j}^{-2} h_{\alpha, i+j} - 2 f_{\alpha, j}^{-3} f_{\alpha, i+2j}.\\
\label{relation2}
[h_i, f_{\alpha, j}^{-1}] &= \alpha(h) f_{\alpha, j}^{-2} f_{\alpha, i+j}. \\
\label{relation3}
[e_{\beta, i}, f_{\alpha, j}^{-1}] &= a f_{\alpha, j}^{-2} e_{\beta-\alpha, i+j} + b f_{\alpha, j}^{-3} e_{\beta-2\alpha, i+2j} + c f_{\alpha, j}^{-4} e_{\beta-3\alpha, i+3j}
\end{align}
for some $a,b,c \in \C$. We adopt the convention $e_{\gamma, i} = 0$ if $\gamma \notin \Phi$ or if $i > n$.
\end{lemma}
\begin{proof}
These can be verified by multiplying by powers of $f_{\alpha, j}$ to obtain an expression which holds in $U$. We show the calculation for (\ref{relation1}); the other relations are proved similarly. Using the relations in $U$ we have:
\begin{align*}
f_{\alpha, j}^3 e_{\alpha, i} = f_{\alpha, j}^2 e_{\alpha, i} f_{\alpha, j} - f_{\alpha, j} h_{\alpha, i+j} f_{\alpha, j} - 2 f_{\alpha, i+2j} f_{\alpha, j}.
\end{align*}
We then multiply on the left by $f_{\alpha, j}^{-3}$ and on the right by $f_{\alpha, j}^{-1}$ to obtain (\ref{relation1}).
\end{proof}

Let $V_\alpha \subseteq U_\alpha$ be the span of the monomials appearing in Lemma~\ref{L:basisLemma}(2) such that $i_j \ge 0$ for some $j = 0,...,n$. Using an argument identical to the one used in the second half of the proof of \cite[Lemma~5.3]{Ch} we see that $V_\alpha$ can be defined symmetrically as the the span of the monomials appearing in Lemma~\ref{L:basisLemma}(3), subject to the condition $i_j \ge 0$ for some $j$. Using an argument identical to {\it loc. cit.} once again we obtain the following.

\begin{lemma}
\label{subbimodulelemma}
$V_\alpha$ is a $U$-$U$-sub-bimodule of $U_\alpha$.  $\hfill \qed$
\end{lemma}

We now consider the $U$-$U$-bimodule
\begin{eqnarray}
S_\alpha := U_\alpha / V_\alpha.
\end{eqnarray}
By Lemma~\ref{L:basisLemma}, we see that $S_\alpha$ has a basis given by
\begin{eqnarray}
\label{e:onebasis}
\{f_{\alpha, 0}^{i_0} \dots f_{\alpha, n}^{i_n} v(k,l,m) \mid \text{ for } v, k, l, m \text{ as per \eqref{e:introduceklm}, \eqref{e:introducev} and } i_j < 0 \text{ for all } j \}
\end{eqnarray}
and another basis given by:
\begin{eqnarray}
\label{e:anotherbasis}
\{ v(k,l,m) f_{\alpha, 0}^{i_0} \dots f_{\alpha, n}^{i_n} \mid \text{ for } v, k, l, m \text{ as per \eqref{e:introduceklm}, \eqref{e:introducev} and } i_j < 0 \text{ for all } j \}
\end{eqnarray}
As a slight abuse of notation we denote an element of $U_\alpha$ and its coset in $S_\alpha$ by the same symbol. 

Now we pick a special automorphism of $\g$. The simple root $\alpha$ which we have fixed throughout this section gives rise to a reflection $s_\alpha \in W = N_G(\h) / \h$. We lift $s_\alpha$ arbitrarily to an element of $N_G(\h)$, which defines an automorphism of $\g$ via the adjoint representation. We denote this automorphism by $\phi_\alpha$. 

Note that $\phi_\alpha$ acts on the root spaces as $s_\alpha$, i.e. it sends $\g_\beta$ to $\g_{s_\alpha(\beta)}$ and preserves $\h$. Furthermore, after rescaling $e_\alpha$ and $f_\alpha$ if necessary, we may assume that:
\begin{align*}
\phi_\alpha(e_\alpha) &= f_\alpha\\
\phi_\alpha(f_\alpha) &= e_\alpha.
\end{align*}

We extend $\phi_\alpha$ to an automorphism of $\g_n$ by the rule $\phi_\alpha(x_i) = \phi_\alpha(x)_i$ for all $x \in \g$.

If $M$ is a left $U$-module, we denote by $\phi_\alpha(M)$ the module obtained by twisting the left action on $M$ by $\phi_\alpha$ and write $\cdot_\alpha$ for this action. More precisely, if $m \in M$ and $u \in U$ then $u \cdot_\alpha m := \phi_\alpha(u) \cdot m$. Similarly, we can similarly twist the action by $\phi_\alpha^{-1}$,  and use notation $\cdot_{\alpha^{-1}}$ in this case.

Let $\mathcal{C}$ be the full subcategory of $U$-mod whose objects are the $\h$-semisimple modules. The proof of the next result is the same as \cite[Lemma~5.9]{Ch}.
\begin{lemma}
\label{extensionlemma}
$\O(\g_n)$ is a Serre subcategory of the category of $\mathcal{C}$. $\hfill \qed$
\end{lemma}

We define a functor $\mathcal{H}$ from $U$-mod to $\mathcal{C}$ by letting $\mathcal{H}(M)$ be the sum of the $\h$-weight spaces in $M$. 

Now we define two endofunctors of $\mathcal{C}$ by setting
\begin{align}
\label{e:introduceTandG}
T_\alpha M &:= \phi_\alpha(S_\alpha \otimes_U M) \\
G_\alpha M &:= \mathcal{H}(\Hom_U(S_\alpha, \phi_\alpha^{-1}(M))).
\end{align}
Note that the action of $U$ on $G_\alpha M$ is given by
$$(u \cdot f)(s) = f(s \cdot u) \text { for } u \in U, \ f \in G_\alpha M, \ s \in S_\alpha.$$
Also note that if $M, N \in \mathcal{C}$ and $\chi\in \Hom_{\mathcal{C}}(M,N)$, then $T_\alpha(\chi): T_\alpha M \rightarrow T_\alpha N$ and $G_\alpha(\chi): G_\alpha M \rightarrow G_\alpha N$ are given by
\begin{eqnarray*}
& & T_\alpha(\chi)(s \otimes m) = s \otimes \chi(m) \text{ for } s \in S_\alpha, \ m \in M,\\ & & G_\alpha(\chi)(\rho) = \chi \circ \rho \text{ for } \rho \in \Hom_U(S_\alpha, \phi_\alpha^{-1}(M))).
\end{eqnarray*}

In order to see these functors are well defined, the only non-trivial check is that $T_\alpha M \in \mathcal{C}$ for any $M \in \mathcal{C}$. For $M \in \mathcal{C}$, we see that $T_\alpha M$ is spanned by $\{f_{\alpha, 0}^{-i_0} \dots f_{\alpha, n}^{-i_n} \otimes w : i_j > 0, w \in \mathcal{H}(M)\}$. For $w \in M^{\lambda}$ and $h \in \h$ we have
\begin{align*}
h \cdot_\alpha (f_{\alpha, 0}^{-i_0} \dots f_{\alpha, n}^{-i_n} \otimes w) &= (s_\alpha(h)f_{\alpha, 0}^{-i_0} \dots f_{\alpha, n}^{-i_n}) \otimes w \\
&= f_{\alpha, 0}^{-i_0} \dots f_{\alpha, n}^{-i_n} (s_\alpha(h) - (i_0 + \dots + i_n) \alpha(h))\otimes w \\
&= f_{\alpha, 0}^{-i_0} \dots f_{\alpha, n}^{-i_n} \otimes (s_\alpha(h) - (i_0 + \dots + i_n) \alpha(h)) w \\
&= f_{\alpha, 0}^{-i_0} \dots f_{\alpha, n}^{-i_n} \otimes (s_\alpha(\lambda) - (i_0 + \dots + i_n)\alpha)(h) w \\
&= (s_\alpha(\lambda) - (i_0 + \dots + i_n)\alpha)(h) (f_{\alpha, 0}^{-i_0} \dots f_{\alpha, n}^{-i_n} \otimes w).
\end{align*}
To summarise, for any $w \in M$ of weight $\lambda$ we have
\begin{align}
\label{tensorweight}
f_{\alpha, 0}^{-i_0} \dots f_{\alpha, n}^{-i_n} \otimes w \in (T_\alpha M)^{s_\alpha(\lambda) - (i_0 + \dots + i_n)\alpha}.
\end{align}
In particular, $T_\alpha M$ is spanned by weight vectors, and so is $\h$-semisimple. The following result can be proven by a calculation almost identical to \cite[Lemma~5.4]{Ch}.

\begin{lemma}
\label{weightfunctionslemma}
Let $g \in \Hom_U(S_\alpha, \phi_\alpha^{-1}(M))$. Then $g$ has weight $\lambda\in \h^*$ if and only if $g(f_{\alpha, 0}^{-i_0} \dots f_{\alpha, n}^{-i_n})$ has weight $\lambda + (i_0 + \dots + i_n)\alpha$ in $\phi^{-1}_\alpha(M)$ for all $i_0,...,i_n \in \Z_{\ge 0}$, which is if and only if this vector has weight $s_\alpha(\lambda) - (i_0 + \dots + i_n) \alpha$ in $M$.
\end{lemma}

We immediately obtain the following consequence.
\begin{corollary}
\label{weightfunctionscor}
Let $g \in \Hom_U(S_\alpha, \phi_\alpha^{-1}(M))$, Then $g$ is a weight vector if and only if $g(f_{\alpha, 0}^{-i_0} \dots f_{\alpha, n}^{-i_n})$ is a weight vector for all $i, j \geq 0$.  In particular, $G_\alpha(M)$ is the direct sum of such vectors $g$.
\end{corollary}

\begin{lemma}
$T_\alpha$ is right exact and $G_\alpha$ is left exact.
\end{lemma}

\begin{proof}
For any module $M \in U\lmod$
$$\phi_\alpha(S_\alpha \otimes_U M) \cong \phi_\alpha(S_\alpha) \otimes_U M$$ and $(\bullet) \otimes_U M$ is right exact, hence $T_\alpha$ is right exact. 

Similarly $G_\alpha$ is a composition of two left exact functors, $\Hom_U(S_\alpha, \phi^{-1}_\alpha(\bullet))$ and $\mathcal{H}$, hence it is also left exact.
\end{proof}

\subsection{Twisting functors between blocks of category $\O$}

Note that the Weyl group action on $\h^*$ extends naturally to an action on $(\h_n^{\ge1})^*$ acting diagonally through the identification $(\h_n^{\ge 1})^* = (\h^*)^{\oplus n}$, and this vector space parameterises the Jordan blocks of $\O(\g_n)$. Retaining the notation of the previous section, we can now precisely state Theorem~\ref{T:introtwisting}.
\begin{theorem}
\label{maintwistingtheorem}
Let $\mu \in (\h_n^{\ge 1})^*$ be such that $\mu(h_{\alpha,n}) \ne 0$. The functors $T_\alpha$ and $G_\alpha$ from \eqref{e:introduceTandG} restrict to functors
\begin{eqnarray*}
T_\alpha &:& \calO^{(\mu)} (\g_n) \longrightarrow \calO^{(s_\alpha(\mu))}(\g_n)\\
G_\alpha &:& \calO^{(s_\alpha(\mu))}(\g_n) \longrightarrow \calO^{(\mu)}(\g_n).
\end{eqnarray*}
These form a quasi-inverse pair of equivalences.
\end{theorem}
The proof of Theorem~\ref{maintwistingtheorem} will be broken down into a series of lemmas, which we record and prove over the course of this section. To be more precise, the theorem will follow directly from Lemmas~\ref{maintwistingparta}, \ref{maintwistingpartb}, \ref{maintwistingpartc}, \ref{maintwistingpartd} and \ref{L:twistingnatural}.

For the rest of the section we keep $\alpha$ fixed and let $\mu \in (\h_n^{\ge 1})^*$ be such that $\mu(h_{\alpha,n}) \neq 0$.

The next result is the first step in the proof of Theorem~\ref{maintwistingtheorem}, and is a generalisation of \cite[Lemma~5.3]{Ch}. The proof given here is an alternative, shorter argument.
\begin{lemma}
\label{maintwistingparta}
$T_\alpha$ restricts to a functor $\calO^{(\mu)}(\g_n) \rightarrow \calO^{(s_\alpha(\mu))}(\g_n)$.
\end{lemma}

\begin{proof}
For $M \in \calO^{(\mu)}(\g_n)$ we let $l(M)$ denote the minimal length of a filtration $0 = M_0 \subseteq M_1 \subseteq \dots \subseteq M_{k-1} \subseteq M_k = M$ such that the sections are highest weight modules (Cf. Lemma \ref{finitefiltration}). We have an exact sequence:
\[0 \rightarrow M_1 \rightarrow M \rightarrow M/M_1 \rightarrow 0\]
and since $T_\alpha$ is right exact, we have an exact sequence
\[T_\alpha M_1 \rightarrow T_\alpha M \rightarrow T_\alpha(M/M_1) \rightarrow 0\]
Using Lemma~\ref{extensionlemma} and the fact that $T_\alpha M_1$ is a highest weight module, we can reduce the claim that $T_\alpha(M) \in \O^{(\mu)}(\g_n)$ to the case where $l(M) = 1$, i.e. $M$ is a highest weight module.


For the rest of the proof we fix $M$ highest weight of weight $\lambda \in \h_n^*$, and let $v \in M$ be a highest weight generator of $M$. We will show that $f_{\alpha, 0}^{-1} \dots f_{\alpha, n}^{-1} \otimes v$ is highest weight and generates $T_\alpha M$, which will complete the proof of the lemma.

To see that the vector is maximal, use \eqref{tensorweight} to see that  $e_{\alpha, i} f_{\alpha, 0}^{-1} \dots f_{\alpha, n}^{-1} \otimes v \in T_\alpha M$ lies in an $\h$-eigenspace which is not a weight of $T_\alpha M$ (for any $i \ge 0$ and $\alpha\in \Phi^+$). To see that it is a genuine highest weight vector, one can use \eqref{relation2} to show that $h_i$ acts via $s_{\alpha_i}(h_i)$.

 To see that $f_{\alpha, 0}^{-1} \dots f_{\alpha, n}^{-1} \otimes v$ generates $T_\alpha M$ we use the fact that $Uv = M$ and that $S_\alpha$ has two bases, \eqref{e:onebasis} and \eqref{e:anotherbasis}, to check that every element of $T_\alpha M$ lies in the submodule generated by the set $\{f_{\alpha, 0}^{-i_0} \dots f_{\alpha, n}^{-i_n} \otimes v \mid i_k > 0\}$. Let $L$ denote the span of this set. Note that it is an $(\sl_2)_n$-module, where $(\sl_2)_n$ is the truncated current algebra on $\sl_2 = \langle e_\alpha, h_\alpha, f_\alpha\rangle$. To complete the current proof we show that $L$ is a simple $(\sl_2)_n$-module.
 
Let $\mathfrak{t}_n \subseteq (\sl_2)_n$ denote the span on $h_{\alpha,0},...,h_{\alpha, n}$ and let $\gamma = \lambda|_{\mathfrak{t}_n} \in \mathfrak{t}_n^*$ give the action on $f_{\alpha, 0}^{-1} \dots f_{\alpha, n}^{-1} \otimes v$. If $M_\gamma$ denotes the Verma module of highest weight $\gamma$ then there is a nonzero homomorphism $M_\gamma \to L$ and the dimensions of the weight spaces are the same. Therefore it remains to show that the Verma module $M_\gamma$ is simple. By \eqref{relation2} we see that $\gamma(h_{\alpha, n}) = (s_\alpha \lambda)(h_{\alpha, n}) \ne 0$ and so we can apply Theorem~\ref{maintheorem} to see that $\O^{(\gamma_{\ge 1})}((\sl_2)_n)$ is equivalent to $\O^{(\gamma_{\ge 1})}(\mathfrak{t}_n)$. In the latter category, all highest weight modules are simple, and it follows that $M_\gamma$ is simple, as required. This completes the proof.
\end{proof}

Consider the abelian Lie algebra $\mathfrak{a} = \langle f_{\alpha, 0}, \dots, f_{\alpha, n} \rangle$ contained in $\g_n$, and let $A := \vspan \{f_{\alpha, 0}^{-i_0} \dots f_\alpha^{-i_n}: i_0, \dots, i_n > 0 \} \subseteq S_\alpha$, which is a $U(\mathfrak{a})$-$U(\mathfrak{a})$-subbimodule of $S_\alpha$. Thanks to the description of the two bases \eqref{e:onebasis} and \eqref{e:anotherbasis} of $S_\alpha$, we have an isomorphism of $U$-$U(\mathfrak{a})$-bimodules
\[S_\alpha \cong U \otimes_{U(\mathfrak{a})} A,\]
and and an isomorphism of $U(\mathfrak{a})$-$U$-bimodules
\[S_\alpha \cong A \otimes_{U(\mathfrak{a})} U.\]
In particular, for any left $U$-module $M$, we have an isomorphism of left $U(\mathfrak{a})$-modules
\begin{align*}
S_\alpha \otimes_U M \cong (A \otimes_{U(\mathfrak{a})} U) \otimes_U M \cong A \otimes_{U(\mathfrak{a})} M.
\end{align*}

Recall that $\mu \in \h_n^*$ satisfies $\mu(h_{\alpha,n}) \neq 0$.
\begin{lemma} \cite[Lemma~5.13]{Ch}
\label{tensorzerolemma}
Let $M \in \calO^{(\mu)}$. Let $m \in M^\lambda \backslash \{0\}$ for some $\lambda \in \h^*$. Then the following are equivalent:
\begin{enumerate}
\item[(a)] In $T_\alpha M$, we have $f_{\alpha, 0}^{-i_0} \dots f_{\alpha, n}^{-i_n} \otimes m = 0$.

\item[(b)] For any vector space $V$ and $U(\mathfrak{a})$-balanced map $\varphi: A \times M \rightarrow V$ we have $\varphi(f_{\alpha, 0}^{-i_0} \dots f_{\alpha, n}^{-i_n}, m) = 0$.

\item[(c)] There exist $m_0, \dots m_n \in M$ such that $m = f_{\alpha, 0}^{i_0} \cdot m_0 + \dots + f_{\alpha, n}^{i_n} \cdot m_n$.
\end{enumerate}
\end{lemma}

The following fact only depends on the existence of inverses for $f_{\alpha, i}$ in $S_\alpha$. We omit the proof.
\begin{lemma} \cite[Lemma~5.15]{Ch}
\label{elementformlemma}
Let $M \in U$-mod. Then any element of $S_\alpha \otimes M$ can be written in the form $f_{\alpha, 0}^{-i_0} \dots f_{\alpha, n}^{-i_n} \otimes m$, for some $i_0, \dots, i_n > 0$ and $m \in M$.
\end{lemma}

\begin{lemma} \cite[Lemma~5.7(b)]{Ch}
\label{maintwistingpartb}
 For any $M \in \calO^{(\mu)}(\g_n)$, the map $\psi_M : M \rightarrow G_\alpha T_\alpha M$ given by $\psi_M(m)(s) = s \otimes m$ is an isomorphism.
\end{lemma}

\begin{lemma} \cite[Lemma~5.16]{Ch}
\label{homextensionlemma}
Let $M$ be a $U$-module, and let $A \subseteq S_\alpha$ and $\mathfrak{a} \subseteq \g$ be as in Lemma \ref{tensorzerolemma}. Let $\varphi: A \rightarrow M$ be a $U(\mathfrak{a})$-homomorphism. Then $\varphi$ extends uniquely to a $U$-homomorphism $\varphi: S_\alpha \rightarrow M$.
\end{lemma}

This next result is a generalisation of \cite[Lemma~5.17]{Ch}, but we provide the proof which is more complicated in the present setting.
\begin{lemma}
\label{homconstruction}
Let $M \in \calO^{(\mu)}(\g_n)$ and let $\mathcal{I}$ be a subset of $\Z^{n+1}$ satisfying:

\begin{enumerate}
\item[(1)] $(i_0, \dots, i_n) \in\mathcal{I}$ whenever any $i_k \leq 0$.

\item[(2)] If $(i_0, \dots, i_n) \in \mathcal{I}$, then $(i_0, \dots, i_k-1, \dots, i_n) \in \mathcal{I}$ for any $0 \leq k \leq n$.
\end{enumerate}
and let $\{m_{i_0, \dots, i_n} \in M : (i_0, \dots, i_n) \in \mathcal{I}\}$ be a collection of elements of $M$ satisfying:
\begin{enumerate}
\item[(i)] $m_{i_0, \dots, i_n} = 0$ whenever any $i_k \leq 0$.

\item[(ii)] $e_{\alpha, k} \cdot m_{i_0, \dots, i_n} = m_{i_0, \dots, i_k-1, \dots, i_n}$ whenever $(i_0, \dots, i_n) \in \mathcal{I}$
\end{enumerate}
Then there exists a $U(\mathfrak{a})$-homomorphism $\varphi: A \rightarrow \phi_\alpha^{-1}(M)$ such that $\varphi(f_{\alpha, 0}^{-i_0} \dots f_{\alpha, n}^{-i_n}) = m_{i_0, \dots, i_n}$, which by Lemma \ref{homextensionlemma} extends to a $U$-homomorphism $\varphi: S_\alpha \rightarrow \phi_\alpha^{-1}(M)$. Moreover, if there exists $\lambda \in \h^*$ such that the weight of $m_{i_0, \dots, i_n}$ is $\lambda - (i_0 + \dots + i_n)\alpha$ for all $(i_0, \dots, i_n) \in \mathcal{I}$, then we can choose $\varphi$ to also be weight and hence in $G_\alpha M$.
\end{lemma}

\begin{proof}
We construct elements $m_{i_0, \dots, i_n}$ for $(i_0, \dots, i_n) \in \Z^{n+1} \backslash \mathcal{I}$ inductively such that the $m_{i_0, \dots, i_n}$ satisfy conditions (i) and (ii) above for any $(i_0, \dots, i_n) \in \Z^{n+1}$. Then observe that (i) and (ii) ensure that setting $\varphi(f_{\alpha, 0}^{-i_0} \dots f_{\alpha, n}^{-i_n}) = m_{i_0, \dots, i_n}$ defines a $U(\mathfrak{a})$-homomorphism $\varphi: A \rightarrow \phi_\alpha^{-1}(M)$, since it is enough to check that $m_{i_0, \dots, i_k - 1, \dots i_n} = f_{\alpha, k} \cdot_{\alpha^{-1}} \varphi(f_{\alpha, 0}^{-i_0} \dots f_{\alpha, n}^{-i_n}) = f_{\alpha, k} \cdot_{\alpha^{-1}} m_{i_0, \dots, i_n} = e_{\alpha, k} \cdot m_{i_0, \dots, i_n}$ for any $0 \leq k \leq n$ and $(i_0, \dots, i_n) \in \Z^{n+1}$.

To construct such $m_{i_0, \dots, i_n}$, let $(i_0, \dots, i_n) \in \Z^{n+1} \backslash \mathcal{I}$ be such that $i_0 + \dots + i_n$ is minimal among elements of $\Z^{n+1} \backslash \mathcal{I}$. Then in particular, $(i_0, \dots, i_k - 1, \dots, i_n) \in \mathcal{I}$ for each $0 \leq k \leq n$. Let $(\sl_2)_n$ be the truncated current algebra on the copy of $\sl_2$ spanned by $e_\alpha, h_\alpha, f_\alpha$ and let $N$ be the $(\mathfrak{sl}_2)_n$-module generated by $\{m_{i_0, \dots, i_k - 1, \dots, i_n}: 0 \leq k \leq n\}$, so $N \in \calO^{(\mu)}((\mathfrak{sl}_2)_n)$. Here we make a slight abuse of notation, identifying $\mu$ with the restriction to $\h_n \cap (\sl_2)_n$. Then we use the following claim to construct $m_{i_0, \dots, i_n} \in N \subseteq M$ such that $e_{\alpha, k} \cdot m_{i_0, \dots, i_n} = m_{i_0, \dots, i_k - 1, \dots, i_n}$.


\begin{claim}
Consider the maps:
\[N \overset{\theta_1}{\longrightarrow} N^{n+1} \overset{\theta_2}{\longrightarrow} N^{\frac{1}{2}n(n+1)}\]
given by $\theta_1(x) = (e_{\alpha, 0} \cdot x, \dots, e_{\alpha, n} \cdot x)$ and $\theta_2(y_0, \dots, y_n) = (e_{\alpha, k} \cdot y_l - e_{\alpha, l} \cdot y_k)_{0 \leq k < l \leq n}$. Then $\ker(\theta_2) = \im(\theta_1)$. Furthermore, if $(y_0, \dots, y_n) \in \ker(\theta_2)$ and $y_k \in N \cap M^\lambda$ for all $0 \leq k \leq n$, then there exists $x \in M^{\lambda - \alpha}$ such that $\theta_1(x) = (y_0, \dots, y_n)$.
\end{claim}


We now prove the claim, which will take several steps. 

First observe that $\theta_2 \circ \theta_1 = 0$, so certainly $\ker(\theta_2) \supseteq \im(\theta_1)$. Hence we only need to show that $\im(\theta_1) \supseteq \ker(\theta_2)$. Throughout the proof of the claim, we write $e_k$ for $e_{\alpha, k}$.

We first deal with the case where $n = 1$ and $N = M_{\gamma}$ is a Verma module. In this case, we consider the restriction of these maps to certain weight spaces in the following way (for any $\lambda \in \h^*$):
\[N^\lambda \overset{\theta_1}{\longrightarrow} (N^{\lambda + \alpha})^2 \overset{\theta_2}{\longrightarrow} N^{\lambda + 2\alpha}\]
Now, either $\dim(N^{\lambda + \alpha}) = 0$, in which case $\im(\theta_1) = \ker(\theta_2)$ automatically, or the dimensions of these weight spaces satisfy $\dim(N^\lambda) = n + 1$, $\dim(N^{\lambda + \alpha}) = n$, and $\dim(N^{\lambda + 2\alpha}) = n - 1$. Hence by considering dimensions and the fact that $\ker(\theta_2) \supseteq \im(\theta_1)$, it is enough to show that $\theta_1$ is injective and $\theta_2$ is surjective. We can compute that $e_1$ acts on the basis vectors $f_1^i f_0^j \otimes 1$ by:
\begin{eqnarray}
\label{e:aneqinthemiddle}
e_1 \cdot (f_1^i f_0^j \otimes 1) = \mu j f_1^i f_0^{j-1} \otimes 1 - j(j-1)f_1^{i+1} f_0^{j-2} \otimes 1
\end{eqnarray}
so by considering $\theta_2(x, 0)$, we see $\theta_2$ is surjective using \eqref{e:aneqinthemiddle} and an inductive argument. We also see that $e_1 \cdot v = 0$ if and only if $v = f_1^i \otimes 1$. Since $\mu \neq 0$, we can show that $e_0 \cdot f_1^i \otimes 1 \neq 0$, so $\theta_1$ is injective as required.

We now deal with the case where $n \geq 1$ and $N = M_{\lambda}$  is a Verma in the $\mu$-Jordan block, so $\lambda|_{\ge 1} = \mu$. We use the following facts, which can be verified by computing the action of $e_n$ on the basis elements $f_0^{i_0} \dots f_n^{i_n} \otimes 1$ of  $M_{\lambda}$ and recalling that $\mu(h_\alpha) \neq 0$:
\begin{enumerate}
\item[(a)] $e_n \cdot N = N$, which uses an inductive argument with \eqref{e:aneqinthemiddle}.
\item[(b)] $e_n \cdot m = 0$ if and only if $m \in \vspan\{f_1^{i_1} \dots f_n^{i_n} \otimes 1: i_1, \dots, i_n \geq 0\}$. The latter is isomorphic as a $U(\langle e_0, \dots, e_{n-1} \rangle)$-module to the Verma module $M_{\gamma}$ over $\vspan\{ e_{\alpha, i}, h_{\alpha,i}, f_{\alpha,i} \mid i=0,...,n-1\} \cong ((\sl_{2})_\alpha)_{n-1}$ where $\gamma(h_{\alpha, i}) := \mu(h_{\alpha, i+1})$.
\item[(c)] Applying (b) repeatedly, we see that for any $k \geq 1$, we have that $e_l \cdot m = 0$ for all $k \leq l \leq n$ if and only if $m \in \vspan\{f_{n-k+1}^{i_{n-k+1}} \dots f_n^{i_n} \otimes 1: i_{n-k+1}, \dots, i_n \geq 0\}$.
\end{enumerate}
To ease notation slightly we will write $\mathfrak{a}_{(i)} = \langle e_0,...,e_i\rangle$, write $(\sl_2)_{(i)}$ for the Lie algebra $((\sl_2)_\alpha)_i$ and $\gamma_{(i)}$ for the character of $\vspan\{h_{\alpha, 0},..., h_{\alpha,i}\}$ given by $\gamma_{(i)}(h_{\alpha,j}) := \mu(h_{\alpha, n-i})$.

Now let $(y_0, \dots, y_n) \in \ker(\theta_2)$, i.e. $e_k \cdot y_l - e_l \cdot y_k = 0$ for all $0 \leq k, l \leq n$. By fact (a), there certainly exists an $x_n$ such that $e_n \cdot x_n = y_n$. We then inductively construct $x_k$ for $2 \leq k \leq n$ such that $e_l \cdot x_k = y_l$ for all $k \leq l \leq n$. Note that the cases $k = 0, 1$ will be dealt with by an additional argument immediately afterwards, and we will then obtain $x_0,...,x_n$ such that $\theta(x_0,...,x_n) = (y_0,...,y_n)$.

Suppose we have constructed $x_{k+1}$ such that $e_l \cdot x_{k+1} = y_l$ for all $k+1 \leq l \leq n$. For any $k+1 \leq l \leq n$, consider $e_l \cdot (e_k \cdot x_{k+1} - y_k) = e_k \cdot (e_l \cdot x_{k+1}) - e_l \cdot y_k = e_k \cdot y_l - e_l \cdot y_k = 0$. Hence by fact (c), we have $e_k \cdot x_{k+1} - y_k \in \vspan\{f_{n-k}^{i_{n-k}} \dots f_n^{i_n} \otimes 1: i_{n-k}, \dots, i_n \geq 0\}$. But as a $U(\mathfrak{a}_{(k)})$-module, this is isomorphic to $M_{\gamma_{(k)}}$, so by fact (a) there exists $x_k' \in \vspan\{f_{n-k}^{i_{n-k}} \dots f_n^{i_n} \otimes 1: i_{n-k}, \dots, i_n \geq 0\}$ such that $e_k \cdot x_k' = e_k \cdot x_{k+1} - y_k$, and by fact (c) $e_l \cdot x_k' = 0$ for all $k+1 \leq l \leq n$. Setting $x_k = x_{k+1} - x_k'$, we see that $e_k \cdot x_k = e_k \cdot x_{k+1} - e_k \cdot x_k' = e_k \cdot x_{k+1} - (e_k \cdot x_{k+1} - y_k) = y_k$, and for $k+1 \leq l \leq n$ we have $e_l \cdot x_k = e_l \cdot x_{k+1} - e_l \cdot x_k' = y_l - 0$. Hence we have constructed $x_k$ with the desired properties.

Now consider $y_1' = e_1 \cdot x_2 - y_1$ and $y_0' = e_0 \cdot x_2 - y_0$. For $2 \leq l \leq n$, we have $e_l \cdot y_1' = e_l \cdot (e_1 \cdot x_2 - y_1) = e_1 \cdot y_1 - e_l \cdot y_1 = 0$, and similarly $e_l \cdot y_0' = 0$, so $y_0', y_1' \in \vspan\{f_{n-1}^{i_{n-1}} f_n^{i_n} \otimes 1: i_{n-1}, i_n \geq 0\}$ which is isomorphic to $M_{\gamma_{(1)}}$ (Verma module over $((\sl_2)_{(1)}$) as a $U(\langle e_0, e_1 \rangle)$-module. In addition, since $(y_0,...,y_n) \in \ker \theta$ we have $e_0 \cdot y_1' - e_1 \cdot y_0' = 0$, so we can apply the case where $N$ is a Verma module and $n=1$, proved earlier, to find $x'$ such that $e_0 \cdot x' = y_0'$, $e_1 \cdot x' = y_1'$, and $e_l \cdot x' = 0$ for $2 \leq l \leq n$. Then setting $x = x_2 - x'$, we have that $\theta_1(x) = (y_0, \dots, y_n)$.

If $N$ is not a Verma module, let $0 = N_0 \subseteq N_1 \subseteq \dots \subseteq N_{k-1} \subseteq N_k = N$ be a filtration of $N$ such that each quotient is a highest weight module. Since $\mu(h_\alpha) \neq 0$, all these highest weight modules must in fact be Verma modules by \cite[7.1]{W}. Then, given $(y_0, \dots, y_n) \in \ker(\theta_2)$, in the quotient $N/N_{k-1}$ we have that there exists $x \in N$ such that $e_l \cdot x + N_{k-1} = y_l + N_{k-1}$ for all $0 \leq l \leq n$. Hence $e_l \cdot x - y_l \in N_{k-1}$ for all $0 \leq l \leq n$. But $e_{l'} \cdot (e_l \cdot x - y_l) = e_l \cdot (e_{l'} \cdot x - y_{l'})$ for all $0 \leq l, l', \leq n$, so by induction on $k$, there exists $x' \in N_{k-1}$ such that $e_l \cdot x' = e_l \cdot x - y_l$ for all $0 \leq l \leq n$. Note that in this final argument we have used the fact that $y_l' := e_l \cdot x - y_l$ gives a collection of elements lying in the kernel of $\theta$, which allows us to apply the inductive hypothesis. Hence $e_l \cdot (x-x') = y_l$ for all $0 \leq l \leq n$, so $\im(\theta_1) \supseteq \ker(\theta_2)$ as required.

Finally, if $y_0, \dots, y_n \in M^\lambda \cap N$, then given $x \in N \subseteq M$ such that $\theta_1(x) = (y_0, \dots, y_n)$, by considering weight spaces we may replace $x$ with its component $x^{\lambda - \alpha}$ in the $\lambda - \alpha$ weight space and $\theta_1(x) = \theta_1(x^{\lambda - \alpha})$, proving the final part of Claim $(\ast)$.

Since $e_{\alpha, k} \cdot m_{i_0, \dots, i_l -1, \dots, i_n} = m_{i_0, \dots, i_k - 1, \dots, i_l - 1, \dots i_n} = e_{\alpha, l} \cdot m_{i_0, \dots, i_k - 1, \dots i_n}$ for all $0 \leq k < l \leq n$, this claim then allows us to pick $m_{i_0, \dots, i_n}$ such that $e_{\alpha, k} \cdot m_{i_0, \dots, i_n} = m_{i_0, \dots, i_k - 1, \dots i_n}$ for all $0 \leq k \leq n$, and if $m_{i_0, \dots, i_k-1, \dots, i_n}$ has weight $\lambda$ for all $0 \leq k \leq n$, then we can choose $m_{i_0, \dots, i_n}$ to have weight $\lambda - \alpha$. Hence applying this inductively, we can construct for all $(i_0, \dots, i_n) \in \Z^{n+1}$ elements $m_{i_0, \dots, i_n} \in M$ satisfying conditions (i) and (ii).

Suppose there exists $\lambda \in \h^*$ such that $\wt(m_{i_0, \dots, i_n}) = \lambda - (i_0 + \dots + i_n)\alpha$ for all $(i_0, \dots, i_n) \in \mathcal{I}$. Then by construction $m_{i_0, \dots, i_n}$ is weight for all $(i_0, \dots, i_n) \in \Z^{n+1}$, so by Corollary \ref{weightfunctionscor}  we have constructed is a weight vector.
\end{proof}

We are now ready to establish one of the main ingredients towards the proof of Theorem~\ref{maintwistingtheorem}.
\begin{lemma}
\label{maintwistingpartc}
For any $N \in \calO^{(s_\alpha(\mu))}(\g_n)$, the map $\epsilon_N : T_\alpha G_\alpha N \rightarrow N$ given by $\epsilon_N(s \otimes g) = g(s)$ is an isomorphism.
\end{lemma}

\begin{proof}

First we show $\epsilon_N$ is a homomorphism. Let $u \in U$, let $s \in S_\alpha$, and let $g \in G_\alpha N$. Then
\[u \cdot \epsilon_N(s \otimes g) = u \cdot g(s) = \phi_\alpha^{-1} \phi_\alpha (u) \cdot g(s) = g(\phi_\alpha(u) \cdot s) = \epsilon_N((\phi_\alpha(u) \cdot s) \otimes g) = \epsilon_N(u \cdot (s \otimes g))\]
so $\epsilon_N$ is certainly a $U$-homomorphism.

Now let $m \in N$ be a weight vector, and choose $a_0, \dots, a_n$ such that $e_{\alpha, k}^{a_k} \cdot m = 0$ for each $k$. Let
\begin{align*}
\mathcal{I} &= \{(i_0, \dots, i_n) : i_k \leq 0 \mbox{ for some } k \mbox{ or } i_k \leq a_k \mbox{ for all } k\} \subseteq \Z^{n+1}\\
m_{i_0, \dots, i_n} &= e_{\alpha, 0}^{a_0 - i_0} \dots e_{\alpha, n}^{a_n - i_n} \cdot m \mbox{ if } i_k \leq a_k \mbox{ for all } k\\
m_{i_0, \dots, i_n} &= 0 \mbox{ otherwise.}
\end{align*}

Then we can use Lemma \ref{homconstruction} to construct $\varphi \in \Hom_U(S_\alpha, \phi_\alpha^{-1}(M))$ which is a weight vector and therefore in $G_\alpha N$ such that $\varphi(f_{\alpha, 0}^{-a_0} \dots f_{\alpha, n}^{-a_n}) = m$. Hence $\epsilon_N(f_{\alpha, 0}^{-a_0} \dots f_{\alpha, n}^{-a_n} \otimes \varphi) = m$, so $\epsilon_N$ is surjective.

We now show $\epsilon_N$ is injective. By Lemma \ref{elementformlemma} any element of $T_\alpha G_\alpha N$ may be written as $f_{\alpha, 0}^{-i_0} \dots f_{\alpha, n}^{-i_n} \otimes g$ for some $g \in G_\alpha N$. Suppose $\epsilon_N(f_{\alpha, 0}^{-i_0} \dots f_{\alpha, n}^{-i_n} \otimes g) = g(f_{\alpha, 0}^{-i_0} \dots f_{\alpha, n}^{-i_n}) = 0$. We may assume $g$ is weight (i.e. an $\h$-eigenvector). If not, we may write $g$ as a sum of $g_\lambda \in (G_\alpha N)^\lambda$, which by considering weight spaces must all satisfy $g_\lambda(f_{\alpha, 0}^{-i_0} \dots f_{\alpha, n}^{-i_n}) = 0$, and then apply the following argument to each $g_\lambda$.

Observe that 
\begin{eqnarray}
\label{e:consistency}
g(f_{\alpha, 0}^{-j_0} \dots f_{\alpha, n}^{-j_n}) = 0 \text{ if either } j_k \leq 0 \text{ for some } 0 \leq k \leq n \text{  or all } j_k \leq i_k.
\end{eqnarray}
Now choose
\begin{align*}
\mathcal{I} &= \Big\{(j_0, \dots, j_n) \in \Z^{n+1} : \begin{array}{c} j_k \leq i_k \mbox{ for all } 0 \leq k \leq n-1 \mbox{ or } j_n  \leq i_n \\ \mbox{ or } j_k \leq 0 \mbox{ for some } 0 \leq k \leq n\end{array} \Big\} \subseteq \Z^{n+1} \\
m_{j_0, \dots, j_n} &= g(f_{\alpha, 0}^{-j_0} \dots f_{\alpha, n}^{-j_n}) \mbox{ if } j_k \leq i_k  \mbox{ for all } 0 \leq k \leq n-1\\
m_{j_0, \dots, j_n} &= 0 \mbox{ if } j_n \leq i_n \mbox{ or } j_k \leq 0 \mbox{ for some } 0 \leq k \leq n
\end{align*}
Then applying Lemma \ref{homconstruction} to this, we construct $g_n \in \Hom_U(S_\alpha, \phi_\alpha^{-1}(M))$ which is a weight vector (and hence in $G_\alpha N$) such that:

\begin{enumerate}
\item[(a)] $g_n(f_{\alpha, 0}^{-j_0} \dots f_{\alpha, n}^{-j_n}) = 0$ if $j_n \leq i_n$

\item[(b)] $g_n(f_{\alpha, 0}^{-j_0} \dots f_{\alpha, n}^{-j_n}) = g(f_{\alpha, 0}^{-j_0} \dots f_{\alpha, n}^{-j_n})$ if $j_k \leq i_k$ for all $0 \leq k \leq n-1$
\end{enumerate}
Note that these two conditions are not mutually exclusive but they are consistent by \eqref{e:consistency}.

By (a), we can define $g_n' \in G_\alpha N$ by setting $g_n'(f_{\alpha, 0}^{-j_0} \dots f_{\alpha, n-1}^{-j_{n-1}} f_{\alpha, n}^{-j_n+i_n}) = g_n(f_{\alpha, 0}^{-j_0} \dots f_{\alpha, n}^{-j_n})$ and so $g_n = f_{\alpha, n}^{i_n} \cdot g_n'$. By (b) we have $(g - g_n)(f_{\alpha, 0}^{-j_0} \dots f_{\alpha, n}^{-j_n}) = 0$ if either some $j_k \leq 0$ or if $j_k \leq i_k$ for all $0 \leq k \leq n-1$.

We now construct $g_s \in G_\alpha N$ inductively by setting
\begin{align*}
\mathcal{I} &= \{(j_0, \dots, j_n) \in \Z^{n+1} : j_k \leq i_k \mbox{ for all } 0 \leq k \leq s-1 \mbox{ or } j_s \leq i_s \mbox{ or } j_k \leq 0 \mbox{ for some } 0 \leq k \leq n\} \subseteq \Z^{n+1} \\
m_{j_0, \dots, j_n} &= (g - g_n - \dots - g_{s+1})(f_{\alpha, 0}^{-j_0} \dots f_{\alpha, n}^{-j_n}) \mbox{ if } j_k \leq i_k  \mbox{ for all } 0 \leq k \leq s-1\\
m_{j_0, \dots, j_n} &= 0 \mbox{ if } j_s \leq i_s \mbox{ or } j_k \leq 0 \mbox{ for some } 0 \leq k \leq n
\end{align*}
The element $g - g_n - \dots - g_{s+1}$ satisfies a condition analogous to \eqref{e:consistency}, which means that the definition of $m_{j_0,...,j_n}$ is consistent.

This $g_s$ satisfies:
\begin{enumerate}
\item[(a)] $g_s(f_{\alpha, 0}^{-j_0} \dots f_{\alpha, n}^{-j_n}) = 0$ if $j_s \leq i_s$

\item[(b)] $g_s(f_{\alpha, 0}^{-j_0} \dots f_{\alpha, n}^{-j_n}) = (g - g_n - \dots - g_{s+1})(f_{\alpha, 0}^{-j_0} \dots f_{\alpha, n}^{-j_n})$ if $j_k \leq i_k$ for all $0 \leq k \leq s-1$
\end{enumerate}
By (a), we can write this $g_s$ in the form $g_s = f_{\alpha, s}^{i_s} \cdot g_s'$ for some $g_s' \in G_\alpha N$. By (b) we have that $(g - g_n - \dots - g_s)(f_{\alpha, 0}^{-j_0} \dots f_{\alpha, n}^{-j_n}) = 0$ if $j_k \leq i_k$ for all $0 \leq k \leq s-1$. Hence we see that $g = f_{\alpha, 0}^{-i_0} \cdot g_0' + \dots + f_{\alpha, n}^{-i_n} \cdot g_n'$ for some $g_0', \dots, g_n' \in G_\alpha N$, so $f_{\alpha, 0}^{-i_0} \dots f_{\alpha, n}^{-i_n} \otimes g = \sum (f_{\alpha, 0}^{-i_0} \dots f_{\alpha, n}^{-i_n}) \otimes f_{\alpha, k}^{i_k} g_k' = 0$. Hence $\epsilon_N$ is injective.
\end{proof}

The proof of this next lemma is almost identical to \cite[Lemma~5.8]{Ch}, and so we skip it for the sake of brevity.
\begin{lemma}
\label{maintwistingpartd}
$G_\alpha$ restricts to a functor $\calO^{(s_\alpha(\mu))}(\g_n) \rightarrow \calO^{(\mu)}(\g_n)$.
\end{lemma}

This next result can be checked by following the argument for \cite[Lemma~5.8(e)]{Ch} verbatim.
\begin{lemma} 
\label{L:twistingnatural}
The transformations $$\psi: \id_{\calO^{(\mu)}(\g_n)} \rightarrow G_\alpha T_\alpha$$ and $$\epsilon: T_\alpha G_\alpha \rightarrow \id_{\calO^{(s_\alpha(\mu))}(\g_n)}$$ are natural.
\end{lemma}

Finally Theorem~\ref{maintwistingtheorem} follows by combining Lemmas~\ref{maintwistingparta}, \ref{maintwistingpartb}, \ref{maintwistingpartc}, \ref{maintwistingpartd} and \ref{L:twistingnatural}.

\section{Composition Multiplicities}
\label{S:compmult}

\subsection{Definition of composition multiplicity}

Let $M \in \calO(\g_n)$. We define the \emph{character} of $M$ to be the function $\ch(M): \h^* \rightarrow \Z_{\geq 0}$ which sends $\lambda$ to $\dimension(M^\lambda)$, and the \emph{support} of $M$, denoted $\supp(M)$, to be the set $\supp(M)  = \{\lambda \in \h^* : M^\lambda \neq 0\}$.

We define the composition multiplicities $k_\lambda(M)$ of $M$ for $\lambda \in \h^*$ using the following result, whose statement and proof is very similar to \cite[Lemma~6.1]{Ch} or \cite[Proposition 8]{MS}.

\begin{lemma} 
\label{compmultdefn}
Let $\mu \in (\h_n^{\ge 1})^*$. For any $M \in \calO^{(\mu)}(\g_n)$, there are unique $\{k_\lambda(M) \in \Z_{\geq 0} : \lambda \in \h_n^*, \lambda_{\geq 1} = \mu\}$, which we refer to as {composition multiplicities}, such that 
\[\ch(M) = \sum k_\lambda(M) \ch(L_{\lambda})\]
\end{lemma}

We give another interpretation of these  which is closer to the notion of composition multiplicities in artinian categories. The following result is a natural generalisation of \cite[Lemma~6.2]{Ch}.

\begin{lemma}
\label{compmultalternatedefn}
Let $\mu \in (\h_n^{\ge 1})^*$ and $M \in \calO^{(\mu)}(\g_n)$. Then there exists $\mathcal{I} \subseteq (\Z_{\geq 0})^n$ and a descending filtration of $M$ indexed by by $\mathcal{I}$ with the lexicographic ordering, such that: 

\begin{enumerate}
\setlength{\itemsep}{5pt}
\item[(a)] If $(i_1, \dots, i_{k-1},  i_k, \dots, i_n) \in \mathcal{I}$ with $i_k > 0$ then $(i_1, \dots, i_{k-1}, i_k-1, 0, \dots, 0) \in \mathcal{I}$.

\item[(b)] Each quotient $M_{(i_1, \dots, i_n-1)}/M_{(i_1, \dots, i_n)}$ is simple.

\item[(c)] The intersection of all the $M_{(i_1, \dots, i_n)}$ is 0.
\end{enumerate}
Furthermore, for any such filtration $L_{\lambda}$ appears as a quotient $M_{(i_1, \dots, i_n-1)}/M_{(i_1, \dots, i_n)}$ precisely $k_\lambda(M)$ times.
\end{lemma}

\begin{proof}
First observe that by Theorems \ref{maintheorem} and \ref{maintwistingtheorem} we may reduce to the case $\mu(\h_n^n) = 0$.

Recall the notation $\g_n^n = \g \otimes t^n \subseteq \g_n$.

First suppose $n = 1$. In this case $M$ has a filtration $M \supseteq \g_1^1 M \supseteq (\g_1^1)^2 M \supseteq \cdots$. Each quotient $(\g_1^1)^i M / (\g_1^1)^{i+1} M$ lies in BGG category $\calO$ for $\g$ and hence has finite length, so this may be refined to a filtration satisfying (a) and (b).

For $n > 1$ consider the filtration $M \supseteq \g_{n}^n M \supseteq (\g_{n}^n)^2 M \supseteq \cdots$. Each quotient $(\g_n^n)^i M / (\g_n^n)^{i+1} M$ lies in the category $\calO^{(\mu)}(\g_{n-1})$, where we identify $\mu$ with an element of $(\h_{n-1}^{\ge 1})^*$ in the obvious manner, so we may set $M_{(i, 0, \dots, 0)} = (\g_{n}^n)^i M$ and argue by induction that this may be refined to a filtration satisfying (a) and (b).

To show this filtration satisfies (c), it is enough to show that whenever $\mu(\h_n^n) = 0$ every $M \in \calO^{(\mu)}(\g_n)$ has the property that $\bigcap_{i \geq 0} (\g_{n}^n)^i M = 0$. By considering weight spaces, this property is preserved by taking quotients and extensions, so it is enough to verify this in the case where $M = M_{\lambda}$ is a Verma module. Since $\lambda_n = \mu_n = 0$ this Verma module is graded: we place a grading on $U(\g_n)$ by setting, for $x\in \g$,
\begin{eqnarray*}
\deg x_{i} = \left\{ \begin{array}{rl} 0 & \text{ if } i=0,...,n-1,\\ 1 & \text{ if } i =n. \end{array} \right.
\end{eqnarray*}
We transfer the induced grading on $U(\n^-_n)$ to $M_\lambda$ via the isomorphism of $\n^-_n$-modules $U(\n^-_n) \cong M_\lambda$. Now $M_\lambda$ is a positively graded $\g_n$-module and $\bigcap_{i \geq 0} (\g_{n}^n)^i M_\lambda$ is contained in the intersection of all graded components, which is zero.

Finally, we observe that $\ch(M) = \sum_{(i_0, \dots, i_n) \in \mathcal{I}, i_n > 0} \ch (M_{i_0, \dots, i_n - 1}/ M_{i_0, \dots, i_n})$, so by uniqueness in Lemma \ref{compmultdefn} the final part of the Lemma holds.
\end{proof}

This result justifies the use of the terminology composition multiplicities, for $k_\lambda(M)$. From now on we use the notation $[M : L_{\lambda}] := k_\lambda(M)$.

\begin{corollary}
\label{compmultequivalence}
The parabolic induction and restriction functors and the twisting functors $T_\alpha, G_\alpha$ preserve composition multiplicities. $\hfill \qed$
\end{corollary}

\subsection{Computation of composition multiplicities for Verma modules}


We now wish to compute the composition multiplicities $[M_\lambda : L_\nu]$. Thanks to Lemma~\ref{generalisedeigenvalues} and Lemma~\ref{compmultalternatedefn} we know that $[M_\lambda : L_\nu] = 0$ unless $\lambda_n = \nu_n$ and so, using a combination of twisting functors and parabolic induction functors, we can reduce to the case $\lambda_n  = \nu_n = 0$, and this is our strategy.

Let $p : \Z_{\ge 0} \Phi^+ \to \Z_{\ge 0}$ be Kostant's partition function. For $\alpha \in \Z \Phi$ we write $\alpha_0 \in \h_{n-1}^*$ to be the element satisfying $\alpha_0(h_i) = 0$ for $i > 0$ and $\alpha_0(h_0) = \alpha(h)$, where $h\in \h$. Also, for any $\lambda \in\h_n^*$ we identify $\lambda_{\le n-1}$ with an element of $\h_{n-1}^*$ in the obvious manner.

\begin{lemma}
\label{L:characterupgrade}
If $M_{\lambda_{\le n-1} - \alpha_0}(\g_{n-1})$ is the Verma module over $\g_{n-1}$ then
$$\ch M_\lambda = \sum_{\alpha \in \Z_{\ge 0} \Phi^+} p(\alpha) \ch (M_{\lambda_{\le n-1} - \alpha_0}(\g_{n-1})).$$ 
\end{lemma}

\begin{proof}
Note that characters can be defined for any semisimple $\h$-module with finite dimensional weight spaces. Let $\C_{\lambda_0}$ be the $\h$-module of weight $\lambda_0$.

Since $M_\lambda \cong U(\n^-_n) \otimes \C_{\lambda_0}$ as $\h$-modules the lemma follows from the facts:
\begin{enumerate}
\item[(i)] The character of $U(\n^-_n)$ is equal to the character of $S(\n^-_n)$.
\item[(ii)] $S(\n^-_n)$ is a free module over $S(\n^-_{n-1})$ and, for $\alpha \in \Z_{\ge 0} \Phi^+$, there are $p(\alpha)$ basis vectors of weight $-\alpha$.
\end{enumerate} 
\end{proof}

\begin{corollary}
\label{zeroblockcompmult}
If $\lambda, \nu\in \h_n^*$ satisfy $\lambda_n = \nu_n = 0$, and $\lambda_{\ge 1} = \nu_{\ge 1}$, and $L_{\nu_{\le n-1}}(\g_{n-1})$ denotes the simple $\g_{n-1}$-module of highest weight $\nu_{\leq n-1}\in \h_n^*$ then
$$[M_{\lambda}: L_{\nu}] = \sum_{\alpha \in \Z_{\ge 0}\Phi^+} p(\alpha) [M_{\lambda_{\leq n-1} - \alpha_0}(\g_{n-1}): L_{\nu_{\le n-1}}(\g_{n-1})].$$
\end{corollary}

\begin{proof}
By Lemma~\ref{simple0modules} we have $\ch L_{\nu} = \ch L_{\nu_{\le n-1}}$ for all $\nu \in \h_n^*$ satisfying $\nu_n = 0$, and so the claim follows from Lemma~\ref{L:characterupgrade}.   
\end{proof}

For $\mu_n \in \h^*$ we write $\g^{\mu_n}_{n-1} := (\g^{\mu_n})_{n-1} = (\g_{n-1})^{\mu_n}$.

\begin{corollary}
Let $\mu \in (\h_n^{\ge 1})^*$ such that $\g^{\mu_n}$ is the Levi factor of a standard parabolic. Then for any $\lambda, \nu \in \h_n^*$ such that 
$\lambda_{\ge 1} = \nu_{\ge 1} = \mu$ we have
\[[M_{\lambda}: L_{\nu}] =\sum_{\alpha \in \mathbb{Z}_{\ge 0}\Phi^+} p(\alpha) [M_{\lambda_{\le n-1} - \alpha_0} (\g^{\mu_n}_{n-1}): L_{\nu_{\leq n-1}}(\g^{\mu_n}_{n-1})].\]
\end{corollary}
\begin{proof}
This follows from Lemma~\ref{L:supportonthecentreequiv} and Theorem~\ref{maintheorem}, Corollary~\ref{compmultalternatedefn}, Corollary~\ref{zeroblockcompmult}. In order to apply these results one should check that the functors appearing in the first two cited results send highest weight modules to highest weight modules of the corresponding weight, and we leave this verification to the reader. 
\end{proof}

To complete the computation of composition multiplicities for all Verma modules we must show that to any Jordan block we can apply twisting functors to obtain a Jordan block $\calO^{(\mu)}$ such that $\g^\mu$ is the Levi factor of a standard parabolic. This is possible thanks to the following result (see \cite[Lemma 2.1]{Ch} for a proof).

\begin{proposition}
\label{P:weylgroupstandardparabolics}
Let $\mu \in \h^*$. Then there exists $w \in W$ and a Levi factor $\mathfrak{l}$ of a standard parabolic $\p$ such that $\g^{w(\mu)} = \mathfrak{l}$. Furthermore, if $w$ is chosen to be of minimal length subject to this condition and $w = s_{\alpha_n} s_{\alpha_{n-1}} \dots s_{\alpha_1}$ is a reduced expression for $w$, then for each $1 \leq i \leq n$, we have $((s_{\alpha_{i-1}} \dots s_{\alpha_1})\mu)(h_{\alpha_i}) \neq 0$. $\hfill\qed$
\end{proposition}

We now define an action $\bullet_n$ of $W$ on $\h^*$ by $w \bullet_n \lambda := w(\lambda + n\rho) - n\rho$ where $\rho$ is half the sum of the positive roots. Note this generalises the dot action of $W$ on $\h^*$, which corresponds to the case $n=1$, and which controls the central characters of $\g_0$ \cite[\textsection 1.9]{H}. We remark that this $n$-dot action appeared in the work of Geoffriau on the centre of the enveloping algebra of $\g_n$ \cite{Ge}.

If $\alpha$ is any simple root, then since $s_\alpha(\alpha) = -\alpha$ and $s_\alpha$ permutes the other positive roots we have $s_\alpha(\rho) = \rho - \alpha$. We then have
\begin{align*}
s_\alpha \bullet_n \lambda &= s_\alpha(\lambda + n\rho) - n\rho \\
&= s_\alpha(\lambda) + n\rho - n\alpha - n\rho \\
&= s_\alpha(\lambda) - n\alpha.
\end{align*}
We extend this to a $W$-action on $\h_n^*$ by setting
$$(s_{\alpha} \bullet_n \lambda)(h_i) = \left\{\begin{array}{cc}
(s_{\alpha}\bullet_n \lambda_i)(h_i) & \text{ for } i = 0\\
(s_{\alpha} \lambda_i)(h_i) & \text{ for } i \ne 0.
\end{array}\right. $$

By our calculations in the proof of Lemma~\ref{maintwistingparta}, the twisting functors $T_\alpha$ take highest weight modules of weight $\lambda \in \h_n^*$ to highest weight modules of weight $s_\alpha \bullet_n \lambda$, and hence take $M_{\lambda}$ to $M_{s_\alpha \bullet_n \lambda}$ and similar for $L_{\lambda}$. Applying Corollary \ref{compmultequivalence} again we obtain the following.

\begin{corollary}
\label{C:reducingthelevel}
Let $\lambda, \nu \in \h_n^*$ such that $\lambda_{\ge 1} = \nu_{\ge 1}$. Let $w \in W$ have a reduced expression $w = s_{\alpha_k} s_{\alpha_{k-1}} \cdots s_{\alpha_1}$ for simple reflections $s_{\alpha_i}$ such that:
\begin{enumerate}
\item[(a)] $\g^{w(\mu_n)}$ is the Levi factor of a standard parabolic.
\item[(b)] For each $1 \leq i \leq n$, we have $((s_{\alpha_{i-1}} \cdots s_{\alpha_1})\mu)(h_{\alpha_i}) \neq 0$.
\end{enumerate}
Then we have
\begin{eqnarray}
\label{e:compmultreduction}
[M_{\lambda}: L_{\nu}] = [M_{w \bullet_n \lambda}: L_{w \bullet_n \nu}] =  \sum_{\alpha \in \mathbb{Z}_{\ge 0}\Phi^+} p(\alpha) [M_{(w \bullet_n \lambda)_{\leq n-1} -\alpha_0}(\g^{w(\lambda_n)}_{n-1}): L_{(w \bullet_n \nu)_{\leq n-1}}(\g^{w(\mu_n)}_{n-1})].
\end{eqnarray}
\end{corollary}

\end{document}